\documentclass[12pt]{article}
\usepackage{amsfonts}
\usepackage{amsmath, graphicx, amsfonts,amssymb, calrsfs}
\usepackage{amsfonts,mathrsfs, color, amsthm,}
\usepackage{bm}
\usepackage[pagebackref]{hyperref}
\usepackage{xcolor}
\usepackage{amsmath, graphicx, amsfonts,amssymb, calrsfs}
\usepackage{amsfonts,mathrsfs, color, amsthm,mathrsfs,float}
\providecommand{\hK}[1]{\|#1\|_{K^*}}
\newcommand{\hKs}[1]{\|#1\|_{K^*}}
\renewcommand{\hK}[1]{\|#1\|_{K}}

\allowdisplaybreaks

\usepackage{enumitem}
\usepackage{hyperref}
\hypersetup{
  colorlinks   = true, %Colours links instead of ugly boxes
  urlcolor     = blue, %Colour for external hyperlinks
  linkcolor    = blue, %Colour of internal links
  citecolor   = red %Colour of citations
  }

 \def\R{\mathbb{R}}

\def\Ksc{\mathcal{K}^{ss}_{(o)}}

\newtheorem{theorem}{Theorem}[section]
\newtheorem{lemma}{Lemma}[section]
\newtheorem{remark}{Remark}[section]
\newtheorem{proposition}{Proposition}[section]
\newtheorem{corollary}{Corollary}[section]

\newtheorem{definition}{Definition}[section]
\newtheorem{theorema}{Theorem}[section]

\def\bt{\begin{theorem}}
\def\et{\end{theorem}}
\def\bl{\begin{lemma}}
\def\el{\end{lemma}}
\def\br{\begin{remark}}
\def\er{\end{remark}}
\def\bc{\begin{corollary}}
\def\ec{\end{corollary}}
\def\bd{\begin{definition}}
\def\ed{\end{definition}}
\def\bp{\begin{proposition}}
\def\ep{\end{proposition}}

\def\RN{\mathbb{R}^N}
\title{
Sharp anisotropic $L^2$-Caffarelli-Kohn-Nirenberg
inequalities associated with the Minkowski functional
}

\author{
Zhenzhen Wei\\[0.4em]
\small Department of Mathematics and Statistics\\
\small Memorial University of Newfoundland\\
\small St. John's, Newfoundland and Labrador A1C 5S7, Canada
}

\date{}

\begin{document}

\maketitle

\begin{abstract}
Let $K\subset \RN$ be a convex body containing the origin in its interior,
and let $\hK{\cdot}$ be its Minkowski functional. In this paper, we develop
an identity-based framework for sharp anisotropic
$L^2$-Caffarelli-Kohn-Nirenberg inequalities associated with the anisotropic
radial derivative
$$
\mathcal R_K(u)(x)=\frac{x\cdot\nabla u(x)}{\hK{x}},
\quad x\in\RN\setminus\{o\}.
$$
A key point of the present work is that $K$ is not assumed to be
origin-symmetric. Consequently, the Minkowski functional $\hK{\cdot}$ need not
be even, and the usual norm-based anisotropic arguments do not apply directly.
The main tools are anisotropic $L^2$-Hardy and
$L^2$-Caffarelli-Kohn-Nirenberg identities with explicit nonnegative
remainders. These identities yield sharp anisotropic
$L^2$-Caffarelli-Kohn-Nirenberg inequalities whose best constants depend on
the parameter region of $(a,b)\in\mathbb R^2$. We also study the attainability
of the sharp constants in a natural completion space and obtain the
corresponding extremal functions. As further consequences, we derive sharp
anisotropic Heisenberg-type uncertainty principles and max-type anisotropic
gradient inequalities. When $K$ is the Euclidean unit ball, our results recover
the classical Euclidean $L^2$ theory; when $K$ is origin-symmetric, they are
consistent with the usual norm-based anisotropic framework. In particular, the
present results extend the sharp $L^2$-Caffarelli-Kohn-Nirenberg theory to general convex bodies
containing the origin in their interiors, for which the Minkowski functional
may be non-even.
\end{abstract}

\smallskip


\smallskip


\smallskip

\noindent\textbf{Email address:} zhenzhenw@mun.ca

\section{Introduction}
The Caffarelli-Kohn-Nirenberg (CKN) inequalities, introduced in \cite{CKN}, constitute a fundamental class of weighted interpolation inequalities in $\mathbb{R}^N$. They unify and extend several cornerstone results in analysis, including the Hardy, Sobolev, and Gagliardo-Nirenberg inequalities, and play a significant role in the study of elliptic and parabolic equations, the calculus of variations, and mathematical physics.

A central theme in the study of functional and geometric inequalities is the development of a sharp theory, encompassing the determination of optimal constants, the characterization of extremal functions, symmetry properties, and stability. Classical milestones include Talenti’s sharp Sobolev inequality \cite{Talenti} and Lieb’s sharp Hardy-Littlewood-Sobolev inequality \cite{Lieb}; see also \cite{CNV,DD,N} for related developments in Sobolev, Gagliardo-Nirenberg, and weighted inequalities. Within the CKN setting, these issues become considerably more intricate due to the strong dependence of optimal constants and extremals on the parameters. Extensive work has been devoted to sharp constants, existence and structure of optimizers, and symmetry-breaking phenomena; see e.g., \cite{CM13,CC09,CW01,CFL21,CC22,Cos08,DE12,DET09,Dong18,FS,Flynn20, LamLu17,WW03}. More recently, refined CKN inequalities and remainder terms have also been investigated; see \cite{ACP,Cazacu,DM23,DT24,Do,DoLam,DN25a,DN25b,Lam19,Lam21,WW03}. In particular, identity-based approaches have revealed precise formulas with nonnegative remainder terms, linking Hardy-type identities to sharp CKN inequalities. Related methodologies include Hardy identities and the Bessel-pair framework; see \cite{DLL,DPH25,GM11,GM13}.

The modern study of stability in functional inequalities originates from a question of Brezis and Lieb \cite{BL85} concerning the sharp Sobolev inequality in $W^{1,2}(\mathbb{R}^N)$: whether the deficit---defined as the difference between the Dirichlet energy and the optimal Sobolev constant times the critical $L^{2^*}$-norm---can be quantitatively controlled by the squared distance to the manifold of optimizers. This question was resolved by Bianchi and Egnell \cite{BE91}, who established the stability estimate
\begin{equation*}
\int_{\mathbb{R}^{N}} |\nabla u|^{2}\,dx
- S_{N}\!\left(
\int_{\mathbb{R}^{N}} |u|^{\frac{2N}{N-2}}\,dx
\right)^{\!\frac{N-2}{N}}
\ge c_{BE}\!
\inf_{U \in E_{Sob}} \!
\int_{\mathbb{R}^{N}} |\nabla(u - U)|^{2}\,dx,
\end{equation*}
where $S_N$ is the optimal Sobolev constant, $E_{Sob}$ denotes the manifold of extremal functions, and $c_{BE}>0$ is a universal constant. This result shows that the Sobolev deficit controls the distance to the set of optimizers in the natural energy norm and is sharp with respect to both the metric and scaling. However, determining explicit values or bounds for stability constants has remained a challenging problem. Significant progress has been made recently by Dolbeault, Esteban, Figalli, Frank, and Loss \cite{DEFFL}, who developed a systematic approach to the Bianchi-Egnell constant, obtained sharp lower bounds as $N \to \infty$, and established global stability for the Gaussian logarithmic Sobolev inequality via gradient flow techniques. More recently, Chen, Lu, and Tang \cite{CLT24,CLT242,CLT243} derived explicit lower bounds for stability constants in Hardy-Littlewood-Sobolev and higher-order fractional Sobolev inequalities, leading to global stability results for Beckner’s logarithmic Sobolev inequality on the sphere. In a related direction, \cite{CLTW} established optimal stability for the Sobolev inequality on the Heisenberg group using the CR Yamabe flow to upgrade local stability to a global result.

In this paper, we focus on the \(L^2\)-subclass of the CKN inequalities: For
$u\in C_c^\infty(\mathbb R^N\setminus\{o\})$ and $(a,b)\in\R^2$, there exists a constant
$C(N,a,b)\ge0$ independent of $u$, such that
\begin{align}\label{f5}
C^2(N,a,b)
\bigg(
\int_{\mathbb R^N}\frac{u^2}{|x|^{a+b+1}}dx
\bigg)^2
\le
\bigg(
\int_{\mathbb R^N}\frac{u^2}{|x|^{2a}}dx
\bigg)
\bigg(
\int_{\mathbb R^N}\frac{|\nabla u|^2}{|x|^{2b}}dx
\bigg).
\end{align}
Here $|x|$ is the Euclidean norm in $\mathbb{R}^N$, $\nabla u$ represents the gradient of $u$, and
$C_c^\infty(\mathbb R^N\setminus\{o\})$ denotes the space of smooth
compactly supported functions on $\mathbb R^N\setminus\{o\}$. This subclass is particularly significant as it contains several classical inequalities as special cases. Specifically, the choice \(a=-1\), \(b=0\) yields the Heisenberg uncertainty principle, while \(a=1\), \(b=0\) gives Hardy's inequality, and \(a=0\), \(b=0\) recovers the Hydrogen uncertainty principle. Thus, the \(L^2\)-CKN inequalities provide a unified framework that connects these fundamental results in modern analysis. 

The study of sharp constants and extremizers for these inequalities has received significant attention. Costa \cite{Cos08} provided elementary, concise proofs for certain subclasses of \(L^2\)-CKN inequalities \eqref{f5} using integration by parts and expanding-the-square techniques. Building on this, Catrina and Costa \cite{CC09} (see also \cite{CFL21}) determined the sharp constants and completely characterized the extremal functions over the full parameter range via variational methods, spherical harmonic decomposition, and Kelvin-type transforms. 

We now recall the results from \cite{CC09, CFL21}, as they provide the Euclidean benchmark for the anisotropic results proved in this paper.

\begin{theorema}\label{Th:CKNI}
Depending on the location of the point \((a,b)\) in the plane, the following cases arise:

\begin{enumerate}
\item In the region \(\mathcal C\), the optimal constant \(C(N,a,b)\) in the \(L^2\)-CKN inequalities \eqref{f5} is
\[
C(N,a,b)=\frac{\left|N-(a+b+1)\right|}{2},
\]
and it is attained by functions of the form
\[
u(x)=D\exp\!\left(\frac{t|x|^{\,b+1-a}}{b+1-a}\right),
\]
where \(D \neq 0\), with \(t<0\) in \(\mathcal C_1\) and \(t>0\) in \(\mathcal C_2\).

\item In the region \(\mathcal D\), the optimal constant in \eqref{f5} is
\[
C(N,a,b)=\frac{\left|N-(3b-a+3)\right|}{2},
\]
and it is attained by functions of the form
\[
u(x)=D|x|^{2(b+1)-N}\exp\!\left(\frac{t|x|^{\,b+1-a}}{b+1-a}\right),
\]
where \(D \neq 0\), with \(t>0\) in \(\mathcal D_1\) and \(t<0\) in \(\mathcal D_2\).

\item Moreover, the optimal constant in \eqref{f5} fails to be attained only along the line \(a=b+1\), in which case
\[
C(N,b+1,b)=\frac{|N-2(b+1)|}{2}.
\]
\end{enumerate}
\end{theorema}
Here,
\[
\begin{aligned}
\mathcal C_1 &:= \big\{(a,b)\in\R^2\big|\ b+1-a>0,\; b\le \tfrac{n-2}{2}\big\},\\
\mathcal C_2 &:= \big\{(a,b)\in\R^2\big|\ b+1-a<0,\; b\ge \tfrac{n-2}{2}\big\},\\
\mathcal C &:= \mathcal C_1 \cup \mathcal C_2,\\
\mathcal D_1 &:= \big\{(a,b)\in\R^2\big|\ b+1-a<0,\; b\leq\tfrac{n-2}{2}\big\},\\
\mathcal D_2 &:= \big\{(a,b)\in\R^2\big|\ b+1-a>0,\; b\geq \tfrac{n-2}{2}\big\},\\
\mathcal D &:= \mathcal D_1 \cup \mathcal D_2.
\end{aligned}
\]

The main purpose of this paper is to study an anisotropic radial counterpart of the weighted $L^2$-CKN inequality \eqref{f5}, in which the Euclidean norm $|x|$ is replaced by the Minkowski functional $\hK{x}$ of a convex body $K$ containing the origin in its interior (see Section \ref{s2} for details), and the Euclidean gradient term $\nabla u$ is replaced by the anisotropic radial derivative $\mathcal R_K(u)(x)=\frac{x\cdot\nabla u(x)}{\hK{x}}$ (see Section \ref{s2} for details). When $K$ is the Euclidean unit ball, $\hK{x}=|x|$ and $\mathcal R_K(u)$ reduces to the usual Euclidean radial derivative. When $K$ is origin-symmetric, $\hK{\cdot}$ is even, and this framework is closely related to Finsler geometry and the Wulff shape; see e.g., \cite{Bellettini}.

Anisotropic functional inequalities have been studied extensively in Finsler and
convex-geometric settings. Let $N\ge2$. In the norm-based Finsler framework, one usually
works with a convex, even, positively one-homogeneous function
$H:\mathbb R^N\to[0,\infty)$, called a Finsler norm. Its polar function is
defined by
$$
H^\circ(x)=\sup_{\xi\ne0}\frac{x\cdot \xi}{H(\xi)},
\quad x\in\mathbb R^N,\ \ \xi\in\mathbb R^N\setminus\{o\}.
$$
Equivalently, $H^\circ$ is the Minkowski functional of an
origin-symmetric convex body. This is the standard anisotropic setting for many
Sobolev, isoperimetric, Trudinger-Moser and Hardy type inequalities.

For instance, one has the following anisotropic Sobolev-type inequality: for
$u\in C_c^\infty(\mathbb R^N)$,
$$
\left(
\int_{\mathbb R^N}|u|^{p^*}dx
\right)^{\frac1{p^*}}
\le
C
\left(
\int_{\mathbb R^N}H(\nabla u)^pdx
\right)^{\frac1p},
\quad
1<p<N,
\quad
p^*=\frac{Np}{N-p},
$$
where $C>0$ is independent of $u$. Such inequalities and the related convex
symmetrization method were developed in \cite{AFTL97,VanSchaftingen}. Sharp
stability results for anisotropic Sobolev and log-Sobolev inequalities were
obtained by Figalli, Maggi and Pratelli \cite{FMP}; see also
\cite{BianchiCianchiGronchi} for recent developments involving anisotropic
symmetrization, convex bodies and isoperimetric inequalities. Related
anisotropic Sobolev-space characterizations were studied by Lam, Maalaoui and
Pinamonti \cite{LMP2019}.

The anisotropic perimeter associated with $H$ is given, for a smooth set
$E\subset\mathbb R^N$, by
$$
P_H(E)=\int_{\partial E}H(\nu_E)d\mathcal H^{N-1},
$$
where $\nu_E$ is the outer unit normal to $\partial E$ and
$\mathcal H^{N-1}$ denotes the $(N-1)$-dimensional Hausdorff measure. The
corresponding anisotropic isoperimetric inequality states that
$$
P_H(E)\ge N\kappa_N^{\frac1N}|E|^{\frac{N-1}{N}},
$$
where
$
\kappa_N=\big|\{x\in\mathbb R^N:H^\circ(x)\le1\}\big|
$
is the volume of the unit Wulff shape associated with $H$, and equality is
attained by Wulff shapes; see e.g.,
\cite{FigalliMaggiPratelli,FMP,FonsecaMuller,Taylor}. Related convex-geometric variational ideas have also appeared recently in
nonlinear potential theory; see, for instance, \cite{LXZZZ26}.

Anisotropic critical exponential inequalities have also been investigated. Lu, Shen, Xue and Zhu \cite{Lu} established weighted anisotropic isoperimetric inequalities and used them to prove the existence of extremals for singular anisotropic Trudinger-Moser inequalities in the Finsler setting. More precisely, using the notation above, for a bounded domain $D\subset\mathbb R^N$ containing the origin in its interior and for $0<\beta<N$, one has the following singular anisotropic Trudinger-Moser inequality:
$$
\sup_{\substack{
u\in W_0^{1,N}(D),
\int_D H(\nabla u)^N dx\le 1}}
\int_D
\exp\Big(\lambda_{N,\beta}|u|^{\frac{N}{N-1}}\Big)
H^\circ(x)^{-\beta}dx
<\infty,
$$
where $W_0^{1,N}(D)$ is the usual Sobolev space with zero boundary trace,
$\beta$ is the singular weight parameter, and $\lambda_{N,\beta}$ is the
corresponding critical constant.

Anisotropic Hardy inequalities can be viewed as special cases of anisotropic
CKN-type inequalities, and sharp anisotropic Hardy inequalities have also been
studied in several Finsler and convex-geometric settings. Della Pietra, di
Blasio and Gavitone \cite{DBG} studied Hardy-type inequalities involving a
smooth Finsler norm and the anisotropic distance to the boundary. More
precisely, let $\Omega\subset\mathbb R^N$ be a domain with nonempty boundary, the anisotropic distance from $x\in\Omega$ to the boundary is defined by
$
d_H(x)=\inf_{y\in\partial\Omega}H^\circ(x-y),
$
and they considered the sharp anisotropic Hardy inequalities associated with $H$ and $d_H$:
$$
\int_\Omega H(\nabla u)^2dx
\ge
C_H(\Omega)
\int_\Omega \frac{u^2}{d_H^2}dx,
\quad u\in H_0^1(\Omega),
$$
where $C_H(\Omega)>0$ denotes the best constant depending on $H$ and $\Omega$.
Mercaldo, Sano and Takahashi \cite{Mercaldo} developed a unified approach to
Finsler Hardy inequalities involving the Finsler distance from a point or from
the boundary, including the sharp constants and non-attainability. In
particular, if $p\ge1$, $p\ne N$, and
$\Omega\subset\mathbb R^N$ is a domain containing the origin, then the following Finsler Hardy inequality associated with $H^\circ$
$$
\bigg|\frac{N-p}{p}\bigg|^p
\int_\Omega \frac{|u|^p}{H^\circ(x)^p}dx
\le
\int_\Omega
\bigg|
\frac{x}{H^\circ(x)}\cdot\nabla u(x)
\bigg|^pdx
$$
holds for $u\in C_c^\infty(\Omega)$ if $1\le p<N$, and for
$u\in C_c^\infty(\Omega\setminus\{o\})$ if $p>N$.

Anisotropic CKN-type inequalities have also been studied in several directions.
Li and Yan \cite{LiYan23} established the anisotropic weighted CKN inequalities with the coordinate-splitting weights involving both
$|x|$ and $|x'|$, where $x=(x',x_N)$ and $x'=(x_1,\ldots,x_{N-1})$. Bao and Chen \cite{BC25} further investigated the Li-Yan anisotropic
CKN-type inequalities in the case $\delta=1$ and $p=2$, including the existence
of extremal functions, symmetry-breaking regions, and symmetry properties of the
extremal functions. 
Shen \cite{Shen25} studied  the sharp anisotropic CKN inequalities and weighted
anisotropic Hardy-Sobolev type inequalities in the Finsler norm setting,
including the sharp constants, existence results, and explicit forms of extremal
functions.

Convex symmetrization and anisotropic rearrangement are important tools in this
anisotropic framework. The convex symmetrization associated with convex bodies
and the corresponding P\'olya-Szeg\"{o} principle were developed in
\cite{AFTL97,Ferone-Volpicelli,VanSchaftingen}; see also
\cite{BianchiCianchiGronchi,Nguyen} for related developments on affine
P\'olya-Szeg\"{o} inequalities, anisotropic symmetrization, convex bodies and
isoperimetric inequalities.  These rearrangement and comparison tools
will be used below to derive the max-type anisotropic gradient inequalities from
the radial derivative inequalities.

To the best of our knowledge, the existing norm-based anisotropic CKN-type
inequalities in the convex body or Finsler setting are usually formulated under
an origin-symmetry or norm assumption. This motivates us to study the anisotropic $L^2$-CKN identities and inequalities associated with the Minkowski functional of a convex body containing the origin in its interior, and the Minkowski functional need not be even.

There are two main difficulties to overcome. The first one comes from the lack of
origin-symmetry of $K$. If $K$ is not origin-symmetric, then the Minkowski functional $\|\cdot\|_K$ is
generally not even. Hence, the symmetric Cauchy-Schwarz inequality
$
|x\cdot y|\le \|x\|_K\cdot\|y\|_{K^*}
$
is no longer available in this form. Instead, one only has the one-sided
anisotropic cases:
$
x\cdot y\le \|x\|_K\cdot\|y\|_{K^*}$ and $
-x\cdot y\le \|x\|_K\cdot\|-y\|_{K^*}.
$
Thus the gradient formulation has to distinguish the two directions
$\nabla u$ and $-\nabla u$. Consequently, one cannot directly obtain a sharp anisotropic CKN inequality with the gradient term by the Cauchy-Schwarz inequality as in the usual norm-based
formulation. This is why the anisotropic radial derivative $$ \mathcal R_K(u)(x)=\frac{x\cdot\nabla u(x)}{\|x\|_K} $$ is the natural object for the sharp anisotropic CKN inequality.
The second one is related to the weighted $L^2$-CKN inequalities. The
sharp constants depend on the position of $(a,b)\in\mathbb R^2$. Hence, one has to choose the appropriate anisotropic identity in each
parameter region and then study the attainability of the sharp constants in a
natural completion space.

Our method combines three ingredients: the anisotropic Hardy and CKN identities involving $\mathcal R_K(u)$, the anisotropic Cauchy-Schwarz inequality, and the anisotropic symmetrization. The identities identify the relevant constants and
the corresponding equality equations; the sharpness of the constants is then
obtained by analyzing the associated extremal functions in a natural completion space. The anisotropic Cauchy-Schwarz inequality and the anisotropic symmetrization are used to derive the anisotropic gradient inequalities from the anisotropic radial-derivative inequalities.

With this strategy, we obtain three main results. The first one consists of
two identities associated with the Minkowski functional. More precisely,
let $K\subset \RN$ be a convex body containing the origin in its interior, $0<R\leq\infty$, $\sigma_K (x) =\frac{x}{\hK{x}}$ for $x\in\RN\setminus\{o\}$, $B_R^K:=\{x\in\mathbb{R}^N:\hK{x}< R\}$
and $A, B, H\in C^1(0,R)$ with $H=AB$.
One has the following  Hardy-type identity for $u\in C_c^\infty(B_R^K\backslash\{o\})$,   
\begin{align*}
&\int_{B_R^K}A^2\big|\mathcal{R}_K(u)\big|^2dx
=\int_{B_R^K}\Big[\mathrm{div}\big(\sigma_K(x)H(\hK{x})\big)-B^2\Big]u^2dx+\int_{B_R^K}\big|A\mathcal{R}_K(u)+Bu\big|^2dx,
\end{align*} and the following CKN-type identity
for $u\in C_c^\infty(B_R^K\backslash\{o\})\backslash\{0\}$, 
\begin{align*}
\bigg(\int_{B_R^K} \big|A\mathcal{R}_K(u)\big|^2dx \bigg)^{\frac{1}{2}}
\bigg(\int_{B_R^K}B^2 u^2dx\bigg)^{\frac{1}{2}}&=\frac{1}{2}\int_{B_R^K}\mathrm{div}\Big(H(\hK{x})\sigma_K(x)\Big)u^2dx \nonumber \\ &+\frac{1}{2}\int_{B_R^K}\Bigg|\frac{\|Bu\|^{\frac{1}{2}}A\mathcal{R}_K(u)}{\|A\mathcal{R}_K(u)\|^{\frac{1}{2}}}+\frac{\|A\mathcal{R}_K(u)\|^{\frac{1}{2}}Bu}{\|Bu\|^{\frac{1}{2}}}\Bigg|^2dx.
\end{align*}
These two identities reduce to the classical Euclidean identities when
$K$ is the Euclidean unit ball. Suitable choices of $A$ and $B$ yield
the anisotropic Hardy identities, the anisotropic Heisenberg identity, and
the weighted anisotropic $L^2$-CKN identities used later to determine the
sharp constants. The second one concerns the sharp anisotropic Heisenberg uncertainty principle.
Let
$
\mathcal C
=
W^{1,2}(\mathbb R^N)
\cap
\big\{
u:
\int_{\mathbb R^N}\hK{x}^2u^2dx<\infty
\big\},
$ where $W^{1,2}(\mathbb R^N)$ is the usual Sobolev space.
Using the Hardy-type and CKN-type identities together with anisotropic symmetrization, one has, for every $u\in\mathcal C$,
$$
\int_{\mathbb R^N}\|-\nabla u\|_{K^*}^2dx
+
\int_{\mathbb R^N}\hK{x}^2u^2dx
\ge
N\int_{\mathbb R^N}u^2dx
$$
and
$$
\bigg(
\int_{\mathbb R^N}\|-\nabla u\|_{K^*}^2dx
\bigg)^{\frac12}
\bigg(
\int_{\mathbb R^N}\hK{x}^2u^2dx
\bigg)^{\frac12}
\ge
\frac N2
\int_{\mathbb R^N}u^2dx.
$$
Both constants are sharp, and the optimizers are Gaussian-type functions generated
by $\hK{x}$.
The third one gives the following sharp anisotropic $L^2$-CKN inequalities: for
$a,b\in\mathbb R$, $u\in C_c^\infty(\mathbb R^N\setminus\{o\})$
\begin{align}\label{t-1}
\bigg(
\int_{\mathbb R^N}
\frac{|\mathcal R_K(u)|^2}{\hK{x}^{2b}}dx
\bigg)^{\frac{1}{2}}
\bigg(
\int_{\mathbb R^N}
\frac{u^2}{\hK{x}^{2a}}dx
\bigg)^{\frac{1}{2}}
\ge
\widetilde{C}(N,a,b)
\int_{\mathbb R^N}
\frac{u^2}{\hK{x}^{a+b+1}}dx,
\end{align}
where $\widetilde{C}(N,a,b)$ is the sharp constant and depends on the position of $(a,b)\in\mathbb R^2$.
In the corresponding regions of $(a,b)$, the sharp constant is given by
$$
\widetilde{C}(N,a,b)
=
\bigg|\frac{N-a-b-1}{2}\bigg|
\quad\text{or}\quad
\widetilde{C}(N,a,b)
=
\bigg|\frac{N-3b+a-3}{2}\bigg|.
$$
We also characterize the optimizers for \eqref{t-1} in the completion space $\mathcal C_{K,a,b}$,
the completion of
$C_c^\infty(\RN\setminus\{o\})$ under the norm
$$
\|u\|_{\mathcal C_{K,a,b}}
=
\bigg(
\int_{\mathbb R^N}
\frac{|\mathcal R_K(u)|^2}{\hK{x}^{2b}}dx
+
\int_{\mathbb R^N}
\frac{u^2}{\hK{x}^{2a}}dx
\bigg)^{\frac{1}{2}}.
$$

When $K$ is the Euclidean
unit ball, our results recover the classical Euclidean $L^2$ identities and
inequalities; when $K$ is origin-symmetric, they are consistent with the usual
norm-based anisotropic framework, while the present formulation also applies to general convex bodies
containing the origin in their interiors, for which the Minkowski functional need not be even.

The paper is organized as follows. In Section \ref{s2}, we collect basic
facts about convex bodies, polar bodies, Minkowski functionals,
anisotropic radial derivatives, and anisotropic symmetrization. In
Section \ref{s3}, we establish anisotropic $L^2$-Hardy and $L^2$-CKN
identities associated with the Minkowski functional. In Section
\ref{s4}, we combine these identities with anisotropic symmetrization to
prove sharp anisotropic Heisenberg-type inequalities and exhibit
Gaussian-type optimizers. In Section \ref{s5}, we establish sharp
anisotropic $L^2$-CKN inequalities, determine their sharp constants in
the relevant parameter regions, and characterize their extremal
functions in the associated completion spaces. As direct consequences,
we derive sharp radial Heisenberg-type inequalities and the corresponding
max-type anisotropic gradient inequalities.

\section{Preliminaries}\label{s2}
Let $N\ge 2$ be a positive integer. Denote by  $\RN$ the $N$-dimensional Euclidean space.  Let $$B_2^N=\{x\in\RN:|x|\leq 1\} \ \ \mathrm{and} \ \ S^{N-1}=\{x\in\RN:|x| = 1\}$$  be the origin-centered unit ball and  unit sphere in $\RN$, respectively, where $|x|=\sqrt{x\cdot x}$ is the Euclidean norm of $x\in \RN$. We use  $x\cdot y$ to represent the inner product of $x,y\in \RN$.

Let $\Omega$ be an open domain in $\RN$, and let
$$
C^\infty(\Omega)
=
\big\{
u:\Omega\to\mathbb R:
u\ \mathrm{has\ continuous\ derivatives\ of\ all\ orders\ on}\ \Omega
\big\}.
$$
That is, for $u\in C^\infty(\Omega)$ and
$\nu=(\nu_1,\nu_2,\dots,\nu_N)$ with each $\nu_i$ being a nonnegative
integer,
\begin{align*}
D^\nu u
&=
\frac{
\partial^{\nu_1+\nu_2+\cdots+\nu_N}u
}{
\partial x_1^{\nu_1}\partial x_2^{\nu_2}\cdots\partial x_N^{\nu_N}
}
\end{align*}
exists and is continuous on $\Omega$. Let
$$
C_c^\infty(\Omega)
=
\big\{
u:u\in C^\infty(\Omega)\ \mathrm{such\ that\ the\ support\ of}\ u\
\mathrm{is\ a\ compact\ subset\ of}\ \Omega
\big\}.
$$
Let $p\ge1$. Denote by $L^p(\RN)$ the space of real-valued functions on $\RN$ with finite $L_p$ norms. Here, the $L_p$ norm for $u\in L^p(\RN)$ is defined by
$$
\|u\|_p=\bigg(\int_{\RN}|u(x)|^pdx\bigg)^{\frac{1}{p}}.
$$ 
For simplicity, let $\|\cdot\|:=\|\cdot\|_2$. Denote by $W^{1,p}(\RN)$ the
(first-order) $L^p$-Sobolev space, meaning that, if $u\in W^{1,p}(\RN)$, then $u\in L^p(\RN)$ and its weak first-order derivative $\nabla u$ satisfies that $$\|\nabla u\|_p=\bigg(\int_{\RN}|\nabla u(x)|^pdx\bigg)^{\frac{1}{p}}<\infty$$
where $\nabla u= (\partial_1 u,\dots,\partial_N u)$ denotes the (weak) gradient of $u$. Hereafter, $\partial_i u$,  for $i=1,\dots,N$, denotes the weak first-order partial derivative of  $u$, which can be formulated by
\begin{align*}
\int_{\RN} u\partial_i\varphi dx
= -\int_{\RN}  \varphi   \partial_i u dx \ \ \mathrm{for\ any}\ \ \varphi\in C_c^\infty(\RN).
\end{align*}

By a convex body, we mean a set $K\subset \RN$ such that it is compact, has non-empty interior, and satisfies
$(1-\lambda)y+\lambda x\in K\ \text{for every}\ x,y\in K\  \text{and}\  \lambda\in[0,1].$
Let $\mathcal{K}$ denote the class of all convex bodies in $\RN$, $\mathcal{K}_{(o)}$ the class of convex bodies containing the origin $o$ in their interiors, and $\mathcal{K}_{(o)}^c$ the subclass of $\mathcal{K}_{(o)}$ consisting of convex bodies that are  origin-symmetric.
For $K\in\mathcal{K}$, define its support function   $h_K:\mathbb{R}^N \to\mathbb{R}$ by
\begin{align*}
h_K(x)=\max\big\{ y\cdot x: y\in K\big\}\ \ \mathrm{ for\ }\ x\in \mathbb{R}^N.
\end{align*} If $\xi\in S^{N-1}$ and $y\in \partial K$ such that $h_K(\xi)=\xi\cdot y$, then we say $\xi$ is an outer unit  normal vector of $\partial K$ at $y$. Note that it can happen  at  some $y\in \partial K$, its outer unit normal vectors may not be unique. A convex body is said to be smooth if each boundary point $y\in \partial K$ has unique outer normal vector. Clearly,   $h_K(rx)=rh_K(x)$ for $x\in\RN$ and $r\geq0$. A
compact convex subset in $\RN$ is uniquely determined by its support function.

Let $K\in\mathcal{K}_{(o)}$. Define its radial function $\rho_K:\RN\setminus\{o\}\to(0,\infty)$  by
$$
\rho_K(x)=\max\{\lambda\geq0:\lambda x\in K\} \ \ \mathrm{for} \ \  x\in\RN\setminus\{o\}.
$$   Clearly,  $\rho_K(rx)=r^{-1}\rho_K(x)$ for $x\in\RN\setminus\{o\}$ and $r>0$.  The \textit{polar body} $K^*$ of $K\in\mathcal{K}_{(o)}$ is defined by
$$K^*=\big\{x\in\RN:x\cdot y\leq1\text{ for }y\in K\big\}.$$
It can be checked that $(K^*)^*=K$ for each $K\in\mathcal{K}_{(o)}$ and
\begin{align*}
\rho_K(x)h_{K^*}(x)=h_K(x)\rho_{K^*}(x)=1\ \text{for \ any}\ x\in\RN\setminus\{o\}.
\end{align*} Define  $\|\cdot\|_K$, the  Minkowski functional  of $K\in\mathcal{K}_{(o)}$,  by
$$\hK{x}=\inf\big\{\lambda\geq 0:x\in\lambda K\big\},$$
where $
{\lambda K}=\big\{\lambda x:x\in K\big\}
$ and $\lambda\ge0$.
It is clear that, for $ K\in\mathcal{K}_{(o)}$ and for $x\in\RN$,
\begin{align}\label{f21}
\hKs{x}=h_K(x) \ \ \mathrm{and} \ \ \hK{x}=h_{K^*}(x).
\end{align}

A function $u:\RN\to\mathbb{R}$ is said to be \textit{radially symmetric} with respect to $K\in\mathcal{K}_{(o)}$, if there exists a function, still denoted by $u$, such that $u(x)=u(\hK{x})$ for any $x\in \RN.$   By  $\partial K$ we mean   the boundary of $K$.
Then, for $K\in\mathcal{K}_{(o)}$, one has
$$
K=\big\{x\in\RN:\hK{x}\leq 1\big\}\ \ \mathrm{and}\ \ \frac{x}{\hK{x}}\in \partial K\ \ \mathrm{for}\ \ x\neq o. $$  Note that $\partial K =\{x\in\RN: \hK{x}=1\}$. Let $\mathcal{L}:(0,\infty)\times \partial K\to\RN\setminus\{o\}$ be  the map given by  $\mathcal{L}(r,\sigma_K (x))=r\sigma_K (x)$, where
$$ r =\hK{x}\ \mathrm{and}\ \ \sigma_K (x) =\frac{x}{\hK{x}}.
$$ 
Note that $x=r\sigma_K(x)$ and $\|\sigma_K (x)\|_K=1$, due to  the positive homogeneity of $\|\cdot\|_K$. Moreover,
\begin{align*} \frac{\partial x}{\partial r}=\sigma_K (x)=\frac{x}{\hK{x}}\ \ \mathrm{and}\ \
\frac{\partial u}{\partial r}=\sum_{i=1}^{N}\frac{\partial u}{\partial {x_i}}\frac{\partial {x_i}}{\partial r}=\nabla u\cdot\frac{x}{\hK{x}}=\nabla u\cdot\sigma_K(x),\end{align*}  for $u$ whose (weak) gradient $\nabla u$ exists.
Let $d\mu_K$ be the finite Borel measure on $\partial K$ induced by the anisotropic
polar formula
$$
\int_{\mathbb R^N}u(x)dx
=
\int_0^\infty\int_{\partial K}u(r\sigma)r^{N-1}d\mu_K(\sigma)dr
$$
for every non-negative measurable function $u.$ 
Denote by $L^p(\partial K,d\mu_K)$ the corresponding $L^p$ space on $\partial K$.
The formula $ \mathcal{R}_K(u) =\nabla u\cdot\sigma_K(x)$ defines the anisotropic radial derivative of $u$ with respect to $K\in\mathcal{K}_{(o)}$.
 On the other hand, for each $x, y\in \RN$, one has $\sigma_K(x)\in \partial K$ and $\sigma_{K^*}(y)\in \partial K^*$. This  further gives  $\sigma_K(x)\cdot \sigma_{K^*}(y)\leq 1$ and hence,  the following anisotropic Cauchy-Schwarz inequality: for any $x, y\in \RN$ and $K\in \mathcal{K}_{(o)}$,
\begin{align}\label{f22-0}
-\hK{x} \cdot \|-y\|_{K^*}\le x \cdot y\leq  \hK{x} \cdot \|y\|_{K^*}.
\end{align}  When $K\in\mathcal{K}_{(o)}^c$, i.e., $K$  is an origin-symmetric convex body in $\RN$,   $\hK{x}$ is a norm induced,  and hence, it holds that
\begin{align*}
|x \cdot y| \leq  \hK{x} \cdot \|y\|_{K^*}.
\end{align*}  

We say that $K\in\mathcal{K}_{(o)}$ is strictly convex if $\partial K$ does not contain a line segment.  Assume that $K\in\mathcal{K}_{(o)}$ is smooth and strictly convex. Then, $K^*$ is also smooth and strictly convex.  By \eqref{f21}, one gets that $\nabla\hK{x}=\nabla h_{K^*}(x)$ is the unique point $y\in \partial K^*$ such that, for any $x\neq o$,  \begin{align}
    x\cdot y=\hK{x} =x\cdot \nabla\hK{x} \ \ \mathrm{and} \ \ \big\|\nabla\|x\|_{K}\big\|_{K^*}=1. \label{relation-x-y}
\end{align} 
Let $\Ksc\subset \mathcal{K}_{(o)}$ be the set of convex bodies which are smooth and strictly convex.  

We next recall the anisotropic symmetrization associated with the convex body $K$.
This symmetrization will be used in Section \ref{s4}. The construction
below is the convex symmetrization with respect to $K$ (see \cite{VanSchaftingen}).

For a measurable set $E\subset \mathbb{R}^N$ with $0<|E|<\infty$,  let $|E|$ stand for Lebesgue measure of $E$,
we can define
$K$-symmetral of $E$ by
$$
E^K:=r_EK,\quad r_E:=\bigg(\frac{|E|}{|K|}\bigg)^{\frac1N}.
$$
Equivalently, up to a set of Lebesgue measure zero,
$$
E^K=\{x\in\mathbb{R}^N:\hK{x}<r_E\}.
$$ Clearly, $|E^K|=|E|$.
Given a function $u\in W^{1,p}(\RN)$, define its distribution
function by
$$
\mu_u(t):=\big|\{x\in\RN:|u(x)|>t\}\big|,
\quad t>0,
$$
and its decreasing rearrangement by
$$
u^*(s):=\inf\{t>0:\mu_u(t)\le s\},
\quad s\in\big(0,|\mathrm{supp}u|\big),
$$
where $\mathrm{supp}u$ refers to the support of $u$.
We define the rearrangement of $u$ associated with $K$ by
$$
u^K(x):=u^*\big(|K|\cdot\hK{x}^N\big), \quad x\in\RN.
$$
The following elementary properties follow directly from the definition of $u^K$
and will be used repeatedly below.
\begin{proposition}
Let $K\in\mathcal K_{(o)}$, and let $u$ be a measurable function on $\mathbb R^N$
such that $\mu_u(t)<\infty$ for $t>0$. Then the following properties hold.
\begin{enumerate}
\item[(i)]  $|u|$ and $u^K$ are equimeasurable, i.e.,
$
\big|\{x\in\mathbb R^N:|u(x)|>t\}\big|
=
\big|\{x\in\mathbb R^N:u^K(x)>t\}\big|$ for $t>0.
$

\item[(ii)] For $t>0$,
$
\{x\in\mathbb R^N:u^K(x)>t\}
=
\{x\in\mathbb R^N:|u(x)|>t\}^K,
$
up to a set of Lebesgue measure zero. 

\item[(iii)] For $p>0$,
$
(|u|^p)^K=(u^K)^p$
a.e. in $\mathbb R^N.
$

\item[(iv)] If $E\subset\mathbb R^N$ is measurable and $0<|E|<\infty$, then
$
(\chi_E)^K=\chi_{E^K}$
a.e. in $\mathbb R^N.
$
\end{enumerate}
\end{proposition}

If, in addition, $K$ is
origin-symmetric, then $\|\cdot\|_K$ is even and the above construction reduces
to the usual convex symmetrization associated with $K$; see
\cite{AFTL97,Ferone-Volpicelli}.

We use the following anisotropic P\'olya-Szeg\"{o} principle for convex
rearrangements.
\begin{lemma}[\cite{Ferone-Volpicelli}]\label{l2}
Let $p\in(1,\infty)$, $K\in\mathcal K_{(o)}$, and $u\in W^{1,p}(\mathbb R^N)$. Then
\begin{align}\label{f54}
\int_{\mathbb{R}^N}\|-\nabla u(x)\|_{K^*}^pdx
\geq\int_{\mathbb{R}^N}\|-\nabla u^K(x)\|_{K^*}^pdx.
\end{align}
Moreover, if $u\ge0$ and
$
\big|\big\{x:|\nabla u^K(x)|=0\big\}\cap\big\{x:0<u^K(x)<\operatorname*{ess\,sup}u\big\}\big|=0,
$
then the equality holds in (\ref{f54}) if and only if there exists $x_0\in\mathbb{R}^N$ such that
$$
u(x)=u^K(x+x_0)\quad\text{a.e. on }\mathbb{R}^N.
$$
\end{lemma}
The following weighted anisotropic Hardy-Littlewood inequality is a
straightforward consequence of the anisotropic Hardy-Littlewood inequality for
convex symmetrization; see e.g., Van Schaftingen
\cite{VanSchaftingen}. Since we need this  weighted form
later, we record it here and provide a proof for completeness.
\begin{proposition} \label{p1}
Let $p>0$, $K\in\mathcal{K}_{(o)}$,  and $u$ be a non-negative measurable function on
$\mathbb R^N$.
\begin{enumerate}
\item[(i)] If $\omega:[0,\infty)\to[0,\infty)$ is non-increasing, then
\begin{align}\label{f39}
\int_{\RN} u(x)^p \omega\big(\hK{x}\big)dx
\le
\int_{\RN} \big(u^K(x)\big)^p \omega\big(\hK{x}\big)dx.
\end{align}
If $\omega$ is strictly decreasing and the equality holds in \eqref{f39}, then $u=u^K$ a.e. on $\RN$.

\item[(ii)] If $v:[0,\infty)\to[0,\infty)$ is non-decreasing, then
\begin{align}\label{f35}
\int_{\RN} u(x)^p v\big(\hK{x}\big)dx
\ge
\int_{\RN} \big(u^K(x)\big)^p v\big(\hK{x}\big)dx.
\end{align}
If $v$ is strictly increasing and the equality holds in \eqref{f35}, then $u=u^K$ a.e. on $\RN$.
\end{enumerate}
\end{proposition}

\begin{proof}
($i$) 
We start from the anisotropic Hardy-Littlewood inequality for convex
symmetrization (see \cite{VanSchaftingen}):
for any nonnegative measurable functions $f$ and $g$ on $\RN$,
\begin{align}\label{f29}
\int_{\RN} f(x)g(x)dx
\le
\int_{\RN} f^K(x)g^K(x)dx.
\end{align}
For $r>0$, let
$
B_r^K:=\big\{x\in\RN:\hK{x}<r\big\},
$ then 
$
\big(\chi_{B_r^K}\big)^K=\chi_{B_r^K}.
$

Taking
$
f=u^p$ and
$
g=\chi_{B_r^K}
$ in \eqref{f29},
then the anisotropic Hardy-Littlewood inequality \eqref{f29} gives
\begin{align}\label{f28}
\int_{B_r^K} u(x)^pdx
\le
\int_{B_r^K} \big(u(x)^p\big)^Kdx=
\int_{B_r^K} \big(u^K(x)\big)^pdx
\end{align}
for $r>0$. 
Indeed, for $s>0$, one has
$
\big\{u^K>s\big\}
=
\big\{u>s\big\}^K$(see \cite{VanSchaftingen}). Then
$$
\big\{\big(u^p\big)^K>s\big\}
=
\big\{u^p>s\big\}^K=
\big\{u>s^{\frac1p}\big\}^K
=
\big\{u^K>s^{\frac1p}\big\}
=
\big\{\big(u^K\big)^p>s\big\},
$$
and hence,
$
\big(u^p\big)^K=\big(u^K\big)^p$
a.e. in $\RN.
$

For $t>0$, let
$
r(t)=\sup\{r\ge0:\omega(r)>t\},
$ with the convention that $r(t)=0$ if the set is empty. 
As $\omega$ is non-increasing, one has
$$
\{x\in\mathbb R^N:\omega(\hK{x})>t\}=\{x\in\mathbb R^N:\hK{x}<r(t)\}=B_{r(t)}^K
$$
up to a null set. 
By the layer-cake representation, one has
\begin{align}\label{f30}
\omega\big(\hK{x}\big)
=
\int_0^\infty \chi_{\big\{y\in\RN:\omega(\|y\|_K)>t\big\}}(x)dt
=
\int_0^\infty \chi_{B_{r(t)}^K}(x)dt.
\end{align}
Therefore, by \eqref{f28}, \eqref{f30} and Tonelli's theorem,
\begin{align*}
\int_{\RN} u(x)^p \omega\big(\hK{x}\big)dx
&=
\int_{\RN} u(x)^p
\bigg(\int_0^\infty \chi_{B_{r(t)}^K}(x)dt\bigg)dx \\
&=
\int_0^\infty \bigg(\int_{B_{r(t)}^K} u(x)^pdx\bigg)dt \\
&\le
\int_0^\infty \bigg(\int_{B_{r(t)}^K} \big(u^K(x)\big)^pdx\bigg)dt \\
&=
\int_{\RN} \big(u^K(x)\big)^p
\bigg(\int_0^\infty \chi_{B_{r(t)}^K}(x)dt\bigg)dx \\
&=
\int_{\RN} \big(u^K(x)\big)^p \omega\big(\hK{x}\big)dx.
\end{align*}
This proves  \eqref{f39}.

Now we prove the equality case.
Let $K\in\mathcal{K}_{(o)}$, and let $E\subset \RN$ be measurable with
$|E|<\infty$. Let
$\omega:[0,\infty)\to[0,\infty)$ be strictly decreasing. If
\begin{equation}\label{claim-1}
\int_E \omega(\hK{x})dx
=
\int_{E^K} \omega(\hK{x})dx,
\end{equation}
then
$
E=E^K$ a.e. in $\RN.
$
Indeed,
since $|E|=|E^K|$, one has
$
|E\setminus E^K|=|E^K\setminus E|.
$ 
Hence,
\begin{align*}
\int_{E^K}\omega(\hK{x})dx-\int_E\omega(\hK{x})dx
&=
\int_{E^K\setminus E}\omega(\hK{x})dx-\int_{E\setminus E^K}\omega(\hK{x})dx\\
&=
\int_{E^K\setminus E}\big(\omega(\hK{x})-\omega(R)\big)dx +
\int_{E\setminus E^K}\big(\omega(R)-\omega(\hK{x})\big)dx.
\end{align*}
Now, if $x\in E^K\setminus E$, hence $\hK{x}<R$ and then
$
\omega(\hK{x})-\omega(R)>0.
$
If $x\in E\setminus E^K$, hence $\hK{x}\ge R$ and thus
$
\omega(R)-\omega(\hK{x})\ge 0.
$
Therefore, if
$$
\int_E \omega(\hK{x})dx
=
\int_{E^K} \omega(\hK{x})dx,
$$
then both nonnegative integrals above must vanish. 
It follows that
$$
\int_{E^K\setminus E}\bigl(\omega(\hK{x})-\omega(R)\bigr)dx=0
$$
implies $|E^K\setminus E|=0$.
On the other hand, for $x\in E\setminus E^K$, one has
$
\omega(R)-\omega(\hK{x})\ge 0,
$
with strict inequality whenever $\hK{x}>R$, i.e., for
$x\in (E\setminus E^K)\setminus \partial E^K$. Hence, one has
$$
\int_{E\setminus E^K}\bigl(\omega(R)-\omega(\hK{x})\bigr)dx=0
$$
implies
$
|(E\setminus E^K)\setminus \partial E^K|=0.
$
Since $\partial E^K$ has Lebesgue measure zero, it follows that $|E\setminus E^K|=0$.
Since
$
E\Delta E^K:=(E\setminus E^K)\cup(E^K\setminus E)
$
and $(E\setminus E^K)\cap (E^K\setminus E)=\varnothing$, we have
$
|E\Delta E^K|=|E\setminus E^K|+|E^K\setminus E|=0.
$
Therefore,
$
E=E^K \ \ \text{a.e. in }\RN.
$

Now we continue the proof of the equality case.
Assume that $\omega$ is strictly decreasing and that equality holds in \eqref{f39}.
For each $t>0$, let
$
F_t:=\{x\in\RN:u(x)>t\}
$
and
$
F_t^K:=\{x\in\RN:u^K(x)>t\}.
$ Then
the layer-cake representation yields
\begin{align*}
\int_{\RN} u(x)^p \omega(\hK{x})dx
&=
p\int_0^\infty t^{p-1}
\bigg(\int_{F_t}\omega(\hK{x})dx\bigg)dt,\\
\int_{\RN} \big(u^K(x)\big)^p \omega(\hK{x})dx
&=
p\int_0^\infty t^{p-1}
\bigg(\int_{F_t^K}\omega(\hK{x})dx\bigg)dt.
\end{align*}
Hence, the equality in \eqref{f39} implies
$$
\int_{F_t}\omega(\hK{x})dx
=
\int_{F_t^K}\omega(\hK{x})dx
\quad\text{for a.e. }t>0.
$$
By \eqref{claim-1},
one obtains
$
F_t=F_t^K$
a.e. in $\RN$
for a.e. $t>0.$
Equivalently,
$
\chi_{F_t}(x)=\chi_{F_t^K}(x)$
and a.e. $x\in\RN$
and a.e. $t>0$.
Using the layer-cake representation, one has
$$
u(x)=\int_0^\infty \chi_{F_t}(x)dt\
\ \text{and}\ \
u^K(x)=\int_0^\infty \chi_{F_t^K}(x)dt.
$$
Therefore,
$
u(x)=u^K(x)$
for a.e. $x\in\RN.
$

\noindent
($ii$)
Let $v:[0,\infty)\to[0,\infty)$ be non-decreasing.
For $m>0$, let
$
v_m(r)=\min\big\{v(r),m\big\}$ and
$w_m(r)=m-v_m(r).
$
Then $v_m\ge0$, $w_m\ge0$, and $w_m$ is non-increasing because $v_m$ is non-decreasing.
Applying part ($i$) to the non-increasing weight $w_m$, then
$$
\int_{\RN} u(x)^p w_m\big(\hK{x}\big)dx
\le
\int_{\RN} \big(u^K(x)\big)^p w_m\big(\hK{x}\big)dx.
$$
Since $u$ and $u^K$ are equimeasurable, i.e.,
$
\big|\{x\in\RN : u(x)>t\}\big|
=
\big|\{x\in\RN : u^K(x)>t\}\big|$ for  $t>0,
$
in particular,
$$
\int_{\RN} u(x)^pdx
=
\int_{\RN} \big(u^K(x)\big)^pdx,
$$
thus,
$$
\int_{\RN} u(x)^p \Big(m-v_m\big(\hK{x}\big)\Big)dx
\le
\int_{\RN} \big(u^K(x)\big)^p \Big(m-v_m\big(\hK{x}\big)\Big)dx
$$
implies
\begin{align}\label{f36}
\int_{\RN} u(x)^p v_m\big(\hK{x}\big)dx
\ge
\int_{\RN} \big(u^K(x)\big)^p v_m\big(\hK{x}\big)dx.
\end{align}
For each fixed $r\ge 0$, the sequence $\{v_m(r)\}_{m=1}^\infty$ is
non-decreasing and converges to $v(r)$ as $m\to\infty$.
Therefore, using the monotone convergence theorem for \eqref{f36}, one gets
$$
\int_{\RN} u(x)^p v\big(\hK{x}\big)dx
\ge
\int_{\RN} \big(u^K(x)\big)^p v\big(\hK{x}\big)dx.
$$
This proves \eqref{f35}.

The equality statement in ($ii$) follows by applying the equality statement in ($i$) to the strictly decreasing weights $m-v_m$ and then letting $m\to\infty$. We omit no essential details, as the argument is identical to the one above.
\end{proof}

\section{Anisotropic \texorpdfstring{$L^2$}{L2}-Hardy and \texorpdfstring{$L^2$}{L2}-CKN identities on \texorpdfstring{$\RN$}{RN}}\label{s3}
In this section, we will study the anisotropic $L^2$-Hardy and $L^2$-CKN identities with respect to the Minkowski functional $\hK{x}$ for $K\in\Ksc$. 

For $0<R\leq\infty$, let $B_R^K:=\{x\in\mathbb{R}^N:\hK{x}< R\}$. Let $u\in C_c^\infty(B_R^K \setminus\{o\})$ and $H\in C^{1}(0,R)$, where $C^{1}(0,R)$ denotes the space of functions whose first derivative exists and is continuous on $(0,R)$.
Assume $H=AB$ with $A,B\in C^{1}(0,R)$. 
Then, for $\alpha>0$, one has
\begin{align*}
&-\!\int_{B_R^K}\!\!\mathrm{div}\big(H(\hK{x})\sigma_K (x)\big)u^2dx =\int_{B_R^K}\! \!\! H(\hK{x})\sigma_K (x)\cdot 2u\nabla udx 
=2\int_{B_R^K}\!\big(\alpha A \sigma_K(x)\cdot\nabla u\big) \cdot\frac{Bu}{\alpha}dx\\
&\quad\quad\quad \quad \quad =-\alpha^2\int_{B_R^K}A^2\big|\sigma_K (x)\cdot\nabla u\big|^2dx-\frac{1}{\alpha^2}\int_{B_R^K}B^2|u|^2dx+\int_{B_R^K}\Big|\alpha A\sigma_K (x)\cdot\nabla u+\frac{Bu}{\alpha}\Big|^2dx.
\end{align*} Using the notation of the anisotropic radial derivative $\mathcal{R}_K(u)$, one can rearrange the above identity as follows. 
\begin{align}
\!\!\!\alpha^2\!\!\!\int_{B_R^K}\!\!\!\!\! A^2\big|\mathcal{R}_K\!(u)\big|^2\! dx\!+\!\frac{1}{\alpha^2}\!\! \int_{B_R^K}\!\!\!\!B^2u^2\!dx\!=\!\!\! \int_{B_R^K}\!\!\!\!\mathrm{div}\big(\sigma_K(x)H(\hK{x})\!\big)u^2dx+\!\!\!\int_{B_R^K}\!\Big|\alpha A\mathcal{R}_K(u)\!+\!\frac{Bu}{\alpha}\Big|^2\!dx.  \label{1i-1}
\end{align} When $K=B_2^N$, identity \eqref{1i-1} reduces to the Euclidean case \cite{Cazacu}.

\subsection{Anisotropic \texorpdfstring{$L^2$}{L2}-Hardy identities on \texorpdfstring{$\RN$}{RN}}\label{s3.1} 
Taking $\alpha=1$ in \eqref{1i-1}, one gets the following basic anisotropic $L^2$-Hardy identity adapted to $\hK{\cdot}$. 
\begin{theorem}\label{Hardy-identity} Let $K\in\Ksc$ and $0<R\le\infty$. Assume that $A,B,H\in C^1(0,R)$ and $H=AB$. Then, for $u\in C_c^\infty(B_R^K\setminus\{o\})$, one has \begin{align*}
\!\!\!\!\int_{B_R^K}A^2\big|\mathcal R_K(u)\big|^2dx\!=\!\!\int_{B_R^K} \bigg[ \operatorname{div}\big(H(\hK{x})\sigma_K(x)\big)\!-\!B^2 \bigg]u^2dx\!+\!\!\int_{B_R^K} \big|A\mathcal R_K(u)\!+\!Bu\big|^2dx. \end{align*} \end{theorem}
We first derive several Hardy-type consequences of Theorem \ref{Hardy-identity}
by choosing suitable functions $A$ and $B$.

\begin{corollary}
Let $K\in\Ksc$. Then, for every $u\in C_c^\infty(\RN\setminus\{o\})$, one has \begin{align}\label{basic-Hardy} 
&\int_{\RN}\big|\mathcal{R}_K(u)\big|^2dx=\Big(\frac{N-2}{2}\Big)^2\int_{\RN}\frac{u^2}{\hK{x}^2}dx+\int_{\RN}\Big|\mathcal{R}_K(u)+\Big(\frac{N-2}{2}\Big)\frac{u}{\hK{x}}\Big|^2dx.  \end{align}  \end{corollary}
\begin{proof} 
Taking $A=1$ and $B=\big(\frac{N-2}{2}\big)\frac{1}{\hK{x}}$ in Theorem \ref{Hardy-identity}, one has $$ \operatorname{div}\big(H(\hK{x})\sigma_K(x)\big)-B(\hK{x})^2 = \Big(\frac{N-2}{2}\Big)^2\frac1{\hK{x}^2}. 
$$ Hence, Theorem \ref{Hardy-identity} gives \eqref{basic-Hardy}. \end{proof}
If $K=B_2^N$, then \eqref{basic-Hardy} reduces to the classical $L^2$-Hardy identity \cite{Cazacu}.

\begin{corollary}
Let $K\in\Ksc$ and $\lambda\in\mathbb R$. Then, for every $u\in C_c^\infty(\RN\setminus\{o\})$, one has 
\begin{align}\label{f60}
\int_{\RN}\!\!\frac{\big|\mathcal{R}_K(u)\big|^2}{\hK{x}^{\lambda}}dx\!=\!\Big(\frac{N\!-\!\lambda\!-\!2}{2}\Big)^2\int_{\RN}\!\frac{u^2}{\hK{x}^{2+\lambda}}dx+\int_{\RN}\!\Bigg|\frac{\mathcal{R}_K(u)}{\hK{x}^{\frac{\lambda}{2}}}+\Big(\frac{N\!-\!\lambda\!-\!2}{2}\Big)\frac{u}{\hK{x}^{\frac{\lambda}{2}+1}}\Bigg|^2dx.
\end{align} 
\end{corollary}
\begin{proof} 
Taking $A=\hK{x}^{-\frac{\lambda}{2}}$ and $B=\frac{N-\lambda-2}{2}\hK{x}^{-\frac{\lambda}{2}-1}$ in Theorem \ref{Hardy-identity}, one obtains $$ \operatorname{div}\big(H(\hK{x})\sigma_K(x)\big)-B(\hK{x})^2 = \Big(\frac{N-\lambda-2}{2}\Big)^2 \frac1{\hK{x}^{\lambda+2}}.$$
Then Theorem \ref{Hardy-identity} gives \eqref{f60}. 
\end{proof}
Identity \eqref{f60} depends on the parameter $\lambda$, where we can use it to derive the improved anisotropic Hardy and Heisenberg-type inequalities with the power-type weights.

\begin{corollary}
Let $K\in\Ksc$. Then, for every $u\in C_c^\infty(\RN\setminus\{o\})$, one has \begin{align}\label{E1-1}
\int_{\RN}\big|\mathcal{R}_K
(u)\big|^2dx=\int_{\RN}(N-\hK{x}^2)u^2dx+\int_{\RN}\big|\mathcal{R}_K(u)+\hK{x}u\big|^2dx.
\end{align}  
\end{corollary} 
\begin{proof} Taking $A=1$ and $B=\hK{x}$ in Theorem \ref{Hardy-identity}, one has $$ \operatorname{div}\big(H(\hK{x})\sigma_K(x)\big)-B(\hK{x})^2 = N-\hK{x}^2.$$ Then Theorem \ref{Hardy-identity} gives \eqref{E1-1}. 
\end{proof}

\begin{corollary}
Let $K\in\Ksc$. Assume that $b+1-a>0$ and $2b\le N-2$. Then, for every $u\in C_c^\infty(\RN\setminus\{o\})$, one has 
\begin{align}\label{f48}
\!\!\!\!\!\!\!\int_{\RN}\!\bigg|\frac{\mathcal{R}_K(u)}{\hK{x}^{b}}\bigg|^2dx\!+\!\!\!\int_{\RN}\!\frac{u^2}{\hK{x}^{2a}}dx\!-\!(N\!-\!1\!-\!a\!-\!b)\!\int_{\RN}\!\frac{u^2}{\hK{x}^{a+b+1}}dx\!=\!\!\!\int_{\RN}\!\bigg|\frac{\mathcal{R}_K(u)}{\hK{x}^{b}}\!+\!\!\frac{u}{\hK{x}^a}\bigg|^2dx. 
\end{align} 
\end{corollary}
\begin{proof} 
Taking $A=\hK{x}^{-b}$ and $B=\hK{x}^{-a}$ in Theorem \ref{Hardy-identity}, one has $$ \operatorname{div}\big(H(\hK{x})\sigma_K(x)\big) = (N-1-a-b)\hK{x}^{-a-b-1}.$$ Thus, Theorem \ref{Hardy-identity} gives \eqref{f48}. 
\end{proof}

\begin{corollary}
Let $K\in\Ksc$. Assume that $b+1-a<0$ and $2b\ge N-2$. Then, for every $u\in C_c^\infty(\RN\setminus\{o\})$, one has 
\begin{align}\label{f76}
\!\!\!\int_{\RN}\!\!\bigg|\frac{\mathcal{R}_K(u)}{\hK{x}^{b}}\bigg|^2dx\!+\!\!\int_{\RN}\!\!\frac{u^2}{\hK{x}^{2a}}dx\!-\!(a\!+\!b\!+\!1\!-\!N)\!\int_{\RN}\!\!\frac{u^2}{\hK{x}^{a+b+1}}dx\!=\!\!\int_{\RN}\!\!\bigg|\!-\!\frac{\mathcal{R}_K(u)}{\hK{x}^{b}}\!+\!\frac{u}{\hK{x}^a}\bigg|^2dx. 
\end{align}  
\end{corollary}
\begin{proof} Taking $A=-\hK{x}^{-b}$ and $B=\hK{x}^{-a}$ in Theorem \ref{Hardy-identity}, one has $$ \operatorname{div}\big(H(\hK{x})\sigma_K(x)\big) = (a+b+1-N)\hK{x}^{-a-b-1}. $$ Then Theorem \ref{Hardy-identity} gives \eqref{f76}. 
\end{proof}

\begin{corollary}
Let $K\in\Ksc$. Assume that $b+1-a<0$ and $2b\leq N-2$. Then, for every $u\in C_c^\infty(\RN\setminus\{o\})$, one has 
\begin{align}\label{f49}
&\int_{\RN}\frac{\big|\mathcal{R}_K(u)\big|^2}{\hK{x}^{2b}}dx+\int_{\RN}\frac{u^2}{\hK{x}^{2a}}dx-(N-3b+a-3)\int_{\RN}\frac{u^2}{\hK{x}^{a+b+1}}dx\notag\\
&=\int_{\RN}\bigg|-\frac{\mathcal{R}_K(u)}{\hK{x}^{b}}+\bigg(\frac{1}{\hK{x}^{a}}-\frac{N-2b-2}{\hK{x}^{b+1}}\bigg)u\bigg|^2dx.
\end{align}  
\end{corollary}
\begin{proof} Taking $A=-\hK{x}^{-b}$ and $B=\hK{x}^{-a}-(N-2b-2)\hK{x}^{-b-1}$ in Theorem \ref{Hardy-identity}, one has $$ \operatorname{div}\big(H(\hK{x})\sigma_K(x)\big) = \frac{N-3b+a-3}{\hK{x}^{a+b+1}}.$$ Then Theorem \ref{Hardy-identity} gives \eqref{f49}. 
\end{proof}

\begin{corollary}
Let $K\in\Ksc$. Assume that $b+1-a>0$ and $2b\geq N-2$. Then, for every $u\in C_c^\infty(\RN\setminus\{o\})$, one has 
\begin{align}\label{f50}
&\int_{\RN}\frac{\big|\mathcal{R}_K(u)\big|^2}{\hK{x}^{2b}}dx+\int_{\RN}\frac{u^2}{\hK{x}^{2a}}dx-(3b-a+3-N)\int_{\RN}\frac{u^2}{\hK{x}^{a+b+1}}dx\notag\\
&=\int_{\RN}\bigg|-\frac{\mathcal{R}_K(u)}{\hK{x}^{b}}+u\bigg(\frac{-1}{\hK{x}^{a}}+\frac{2b+2-N}{\hK{x}^{b+1}}\bigg)\bigg|^2dx.
\end{align}
\end{corollary}
\begin{proof} Taking $A=-\hK{x}^{-b}$ and $B=-\hK{x}^{-a}+(2b+2-N)\hK{x}^{-b-1}$ in Theorem \ref{Hardy-identity}, one has $$ \operatorname{div}\big(H(\hK{x})\sigma_K(x)\big) = \frac{3b+3-a-N}{\hK{x}^{a+b+1}}.$$ Then Theorem \ref{Hardy-identity} gives \eqref{f50}. \end{proof}

\subsection{Anisotropic \texorpdfstring{$L^2$}{L2}-CKN identities on \texorpdfstring{$\RN$}{RN}} In this subsection, we derive some anisotropic $L^2$-CKN identities on $\RN$. To this end, let us optimize the left hand side in \eqref{1i-1} by letting   $\alpha>0$ such that
\begin{align*} 
\alpha^2\|A\mathcal{R}_K(u)\|^2=
\alpha^2\int_{B_R^K}A^2\big|\mathcal{R}_K(u)\big|^2dx=\frac{1}{\alpha^2}\int_{B_R^K}B^2u^2dx
=\frac{\|Bu\|^2}{\alpha^2} .\end{align*} Hence,  $\alpha=\big(\frac{\|Bu\|}{\|A\mathcal{R}_K(u)\|}\big)^{\frac{1}{2}}$. Thus, one has the following anisotropic $L^2$-CKN identities.
\begin{theorem}\label{CKN-identity} Let $K\in\mathcal K^{ss}_{(o)}$ and $0<R\le\infty$. Assume that $A,B,H\in C^1(0,R)$ and $H=AB$. Then, for $u\in C_c^\infty(B_R^K\setminus\{o\})\setminus\{0\}$ such that $$ \|A\mathcal R_K(u)\|>0 \quad\text{and}\quad \|Bu\|>0, $$ one has \begin{align}\label{li-2}
\!\!\!\!\!\!\bigg(\int_{B_R^K} \big|A\mathcal{R}_K(u)\big|^2\!dx \bigg)^{\frac{1}{2}}
\bigg(\int_{B_R^K}B^2 u^2dx\bigg)^{\frac{1}{2}}&\!=\!\frac{1}{2}\int_{B_R^K}\mathrm{div}\Big(H(\hK{x})\sigma_K(x)\Big)u^2dx \nonumber \\ &\!+\frac{1}{2}\int_{B_R^K}\Bigg|\frac{\|Bu\|^{\frac{1}{2}}A\mathcal{R}_K(u)}{\|A\mathcal{R}_K(u)\|^{\frac{1}{2}}}\!+\!\frac{\|A\mathcal{R}_K(u)\|^{\frac{1}{2}}Bu}{\|Bu\|^{\frac{1}{2}}}\Bigg|^2dx.
\end{align} \end{theorem}
We now derive several consequences of Theorem \ref{CKN-identity} by choosing
suitable functions $A$ and $B$.

\begin{corollary}\label{cor:Heisenberg-CKN} Let $K\in\Ksc$. Then, for every $u\in C_c^\infty(\RN\setminus\{o\})\setminus\{0\}$ such that $$ \|\mathcal R_K(u)\|>0 \quad\text{and}\quad \|\hK{x}u\|>0, $$ one has \begin{align}\label{c3} &\bigg(\int_{\RN}\big|\mathcal R_K(u)\big|^2dx\bigg)^{\frac12} \bigg(\int_{\RN}\hK{x}^2u^2dx\bigg)^{\frac12} \nonumber\\ &=\frac N2\int_{\RN}u^2dx+ \frac12 \int_{\RN} \Bigg| \frac{\big\|\hK{x}u\big\|^{\frac12}\mathcal R_K(u)} {\|\mathcal R_K(u)\|^{\frac12}} + \frac{\|\mathcal R_K(u)\|^{\frac12}\hK{x}u} {\|\hK{x}u\|^{\frac12}} \Bigg|^2dx. 
\end{align} Consequently, 
\begin{align}\label{radial-Heisenberg-CKN-ineq} \bigg(\int_{\RN}\big|\mathcal R_K(u)\big|^2dx\bigg)^{\frac12} \bigg(\int_{\RN}\hK{x}^2u^2dx\bigg)^{\frac12} \ge \frac N2\int_{\RN}u^2dx. 
\end{align} 
\end{corollary}
\begin{proof} Taking $A=1$ and $B=\hK{x}$ in Theorem \ref{CKN-identity}, one has $$  \operatorname{div}\big(H(\hK{x})\sigma_K(x)\big)=N.$$ Thus, \eqref{c3} follows directly from \eqref{li-2}. Inequality \eqref{radial-Heisenberg-CKN-ineq} follows from the nonnegativity of the remainder term. 
\end{proof}
For inequality \eqref{radial-Heisenberg-CKN-ineq}, we will prove a similar inequality with $|\mathcal{R}_K(u)|$ replaced by $\|-\nabla u\|_{K^*}$, and in particular verify whether the constant $\frac{N}{2}$ is a sharp constant in Section \ref{s4}. 

\begin{corollary}\label{cor:A1-CKN} Let $K\in\Ksc$. Assume that $b+1-a>0$ and $2b\le N-2$. Then, for every $u\in C_c^\infty(\RN\setminus\{o\})\setminus\{0\}$ such that $$ \bigg\|\frac{\mathcal R_K(u)}{\hK{x}^{b}}\bigg\|>0 \quad\text{and}\quad \bigg\|\frac{u}{\hK{x}^{a}}\bigg\|>0, $$ one has 
\begin{align}\label{f55}
\bigg(\int_{\RN}\frac{\big|\mathcal{R}_K(u)\big|^2}{\hK{x}^{2b}}dx\bigg)^{\frac{1}{2}}\bigg(\int_{\RN}\frac{u^2}{\hK{x}^{2a}}dx\bigg)^{\frac{1}{2}}&=\frac{N-a-b-1}{2}\int_{\RN}\frac{u^2}{\hK{x}^{a+b+1}}dx\nonumber\\
&+\frac{1}{2}\int_{\RN}\Bigg|\frac{\big\|\frac{u}{\hK{x}^{a}}\big\|^{\frac{1}{2}}\mathcal{R}_K(u)}{\big\|\frac{\mathcal{R}_K(u)}{\hK{x}^{b}}\big\|^{\frac{1}{2}}\hK{x}^{b}}+\frac{\big\|\frac{\mathcal{R}_K(u)}{\hK{x}^{b}}\big\|^{\frac{1}{2}}u}{\big\|\frac{u}{\hK{x}^{a}}\big\|^{\frac{1}{2}}\hK{x}^{a}}\Bigg|^{2}dx.
\end{align} Consequently, 
\begin{align}\label{A1-CKN-ineq} \bigg(\int_{\RN}\frac{\big|\mathcal R_K(u)\big|^2}{\hK{x}^{2b}}dx\bigg)^{\frac12} \bigg(\int_{\RN}\frac{u^2}{\hK{x}^{2a}}dx\bigg)^{\frac12} \ge \frac{N-a-b-1}{2} \int_{\RN}\frac{u^2}{\hK{x}^{a+b+1}}dx . 
\end{align} 
\end{corollary}
\begin{proof} Taking $A=\hK{x}^{-b}$ and $B=\hK{x}^{-a}$ in Theorem \ref{CKN-identity}, one has $$ \operatorname{div}\big(H(\hK{x})\sigma_K(x)\big) = (N-1-a-b)\hK{x}^{-a-b-1}. $$ Thus, \eqref{f55} follows from \eqref{li-2}. Since the last integral in \eqref{f55} is nonnegative, hence, \eqref{A1-CKN-ineq} follows. 
\end{proof}

\begin{corollary}\label{cor:A2-CKN} Let $K\in\Ksc$. Assume that $b+1-a<0$ and $2b\ge N-2$. Then, for every $u\in C_c^\infty(\RN\setminus\{o\})\setminus\{0\}$ such that $$ \bigg\|\frac{\mathcal R_K(u)}{\hK{x}^{b}}\bigg\|>0 \quad\text{and}\quad \bigg\|\frac{u}{\hK{x}^{a}}\bigg\|>0, $$ one has 
\begin{align}\label{f75}
\!\!\!\!\!\bigg(\!\int_{\RN}\frac{\big|\mathcal{R}_K(u)\big|^2}{\hK{x}^{2b}}dx\!\bigg)^{\frac{1}{2}}\bigg(\!\int_{\RN}\frac{u^2}{\hK{x}^{2a}}dx\!\bigg)^{\frac{1}{2}}&\!\!=\!\frac{a+b+1-N}{2}\int_{\RN}\frac{u^2}{\hK{x}^{a+b+1}}dx\nonumber\\
&\!\!+\!\frac{1}{2}\int_{\RN}\Bigg|-\frac{\big\|\frac{u}{\hK{x}^{a}}\big\|^{\frac{1}{2}}\mathcal{R}_K(u)}{\big\|\frac{\mathcal{R}_K(u)}{\hK{x}^{b}}\big\|^{\frac{1}{2}}\hK{x}^{b}}+\frac{\big\|\frac{\mathcal{R}_K(u)}{\hK{x}^{b}}\big\|^{\frac{1}{2}}u}{\big\|\frac{u}{\hK{x}^{a}}\big\|^{\frac{1}{2}}\hK{x}^{a}}\Bigg|^{2}dx.
\end{align} 
 Consequently, 
\begin{align}\label{A2-CKN-ineq} \bigg(\int_{\RN}\frac{\big|\mathcal R_K(u)\big|^2}{\hK{x}^{2b}}dx\bigg)^{\frac12} \bigg(\int_{\RN}\frac{u^2}{\hK{x}^{2a}}dx\bigg)^{\frac12} \ge \frac{a+b+1-N}{2} \int_{\RN}\frac{u^2}{\hK{x}^{a+b+1}}dx . 
\end{align} 
\end{corollary}
\begin{proof} Taking $A=-\hK{x}^{-b}$ and $B=\hK{x}^{-a}$ in Theorem \ref{CKN-identity}, one has $$ \operatorname{div}\big(H(\hK{x})\sigma_K(x)\big) = (a+b+1-N)\hK{x}^{-a-b-1}.$$ Thus, equality \eqref{f75} follows from \eqref{li-2}. Since the last integral in \eqref{f75} is nonnegative, then \eqref{A2-CKN-ineq} follows. 
\end{proof}
We revisit inequalities \eqref{A1-CKN-ineq} and \eqref{A2-CKN-ineq} in Section \ref{s5}, and prove that the constant $\big|\frac{N-a-b-1}{2}\big|$ is a sharp constant in the corresponding parameter region.

\begin{corollary}\label{cor:B1-CKN} Let $K\in\Ksc$. Assume that $b+1-a<0$ and $2b\leq N-2$. Then, for every $u\in C_c^\infty(\RN\setminus\{o\})\setminus\{0\}$ such that $$ \bigg\|\frac{\mathcal R_K(u)}{\hK{x}^{b}}\bigg\|>0 \quad\text{and}\quad \bigg\|\frac{u}{\hK{x}^{a}}\bigg\|>0, $$ one has 
\begin{align}\label{f56}
&\bigg(\int_{\RN}\frac{\big|\mathcal{R}_K(u)\big|^2}{\hK{x}^{2b}}dx\bigg)^{\frac{1}{2}}\bigg(\int_{\RN}\frac{u^2}{\hK{x}^{2a}}dx\bigg)^{\frac{1}{2}} -\frac{N-3b+a-3}{2}\int_{\RN}\frac{u^2}{\hK{x}^{a+b+1}}dx\nonumber \\
&=\frac{1}{2}\int_{\RN}\Bigg|
\frac{\big\|\frac{u}{\hK{x}^{a}}\big\|^{\frac{1}{2}}\mathcal{R}_K(u)}{\big\|\frac{\mathcal{R}_K(u)}{\hK{x}^{b}}\big\|^{\frac{1}{2}}\hK{x}^{b}} 
+\Bigg(\frac{\big\|\frac{u}{\hK{x}^{a}}\big\|^{\frac{1}{2}}}{\big\|\frac{\mathcal{R}_K(u)}{\hK{x}^{b}}\big\|^{\frac{1}{2}}}\frac{(N-2b-2)}{\hK{x}^{b+1}}
-\frac{\big\|\frac{\mathcal{R}_K(u)}{\hK{x}^{b}}\big\|^{\frac{1}{2}}}{\big\|\frac{u}{\hK{x}^{a}}\big\|^{\frac{1}{2}}\hK{x}^{a}} \Bigg)u\Bigg|^2dx.
\end{align}
Consequently, 
\begin{align}\label{B1-CKN-ineq} \bigg(\int_{\RN}\frac{\big|\mathcal R_K(u)\big|^2}{\hK{x}^{2b}}dx\bigg)^{\frac12} \bigg(\int_{\RN}\frac{u^2}{\hK{x}^{2a}}dx\bigg)^{\frac12} \ge \frac{N-3b+a-3}{2} \int_{\RN}\frac{u^2}{\hK{x}^{a+b+1}}dx. 
\end{align} 
\end{corollary}
\begin{proof} Taking $A=-\hK{x}^{-b}$ and $B=\hK{x}^{-a}-(N-2b-2)\hK{x}^{-b-1}$ in Theorem \ref{CKN-identity}, one has $$\operatorname{div}\big(H(\hK{x})\sigma_K(x)\big) = \frac{N-3b+a-3}{\hK{x}^{a+b+1}}.$$ Thus,  equality \eqref{f56} follows from  equality \eqref{li-2}. Since the right hand side of equality \eqref{f56} is nonnegative, then inequality \eqref{B1-CKN-ineq} follows. 
\end{proof}

\begin{corollary}\label{cor:B2-CKN} Let $K\in\Ksc$. Assume that $b+1-a>0$ and $2b\geq N-2$. Then, for every $u\in C_c^\infty(\RN\setminus\{o\})\setminus\{0\}$ such that $$ \bigg\|\frac{\mathcal R_K(u)}{\hK{x}^{b}}\bigg\|>0 \quad\text{and}\quad \bigg\|\frac{u}{\hK{x}^{a}}\bigg\|>0, $$ one has 
\begin{align}\label{f61}
&\bigg(\int_{\RN}\frac{\big|\mathcal{R}_K(u)\big|^2}{\hK{x}^{2b}}dx\bigg)^{\frac{1}{2}}\bigg(\int_{\RN}\frac{u^2}{\hK{x}^{2a}}dx\bigg)^{\frac{1}{2}} -\frac{3b-a+3-N}{2}\int_{\RN}\frac{u^2}{\hK{x}^{a+b+1}}dx\nonumber \\
&=\frac{1}{2}\int_{\RN}\Bigg|
\frac{\big\|\frac{u}{\hK{x}^{a}}\big\|^{\frac{1}{2}}\mathcal{R}_K(u)}{\big\|\frac{\mathcal{R}_K(u)}{\hK{x}^{b}}\big\|^{\frac{1}{2}}\hK{x}^{b}} 
+ \Bigg(\frac{\big\|\frac{u}{\hK{x}^{a}}\big\|^{\frac{1}{2}}}{\big\|\frac{\mathcal{R}_K(u)}{\hK{x}^{b}}\big\|^{\frac{1}{2}}}\frac{(N-2b-2)}{\hK{x}^{b+1}}
+\frac{\big\|\frac{\mathcal{R}_K(u)}{\hK{x}^{b}}\big\|^{\frac{1}{2}}}{\big\|\frac{u}{\hK{x}^{a}}\big\|^{\frac{1}{2}}\hK{x}^{a}}\Bigg)u\Bigg|^2dx.
\end{align} 
Consequently, 
\begin{align}\label{B2-CKN-ineq} \bigg(\int_{\RN}\frac{\big|\mathcal R_K(u)\big|^2}{\hK{x}^{2b}}dx\bigg)^{\frac12} \bigg(\int_{\RN}\frac{u^2}{\hK{x}^{2a}}dx\bigg)^{\frac12} \ge \frac{3b-a+3-N}{2} \int_{\RN}\frac{u^2}{\hK{x}^{a+b+1}}dx . 
\end{align} 
\end{corollary}
\begin{proof} Taking $A=-\hK{x}^{-b}$ and $B=-\hK{x}^{-a}+(2b+2-N)\hK{x}^{-b-1}$ in Theorem \ref{CKN-identity}, one has $$ \operatorname{div}\big(H(\hK{x})\sigma_K(x)\big) = \frac{3b+3-a-N}{\hK{x}^{a+b+1}}. $$ Thus, equality \eqref{f61} follows from equality \eqref{li-2}. Since the right hand side of equality \eqref{f61} is nonnegative, then inequality \eqref{B2-CKN-ineq} follows. 
\end{proof}

The following figure provides a clear view of the regions for $(a, b)\in \R^2$, where the above identities exhibit different formulas, and in fact the best constants of the inequalities depend on where $(a, b)$ lies in. 
\begin{figure}[H]
\centering
\includegraphics[width=0.5\textwidth]{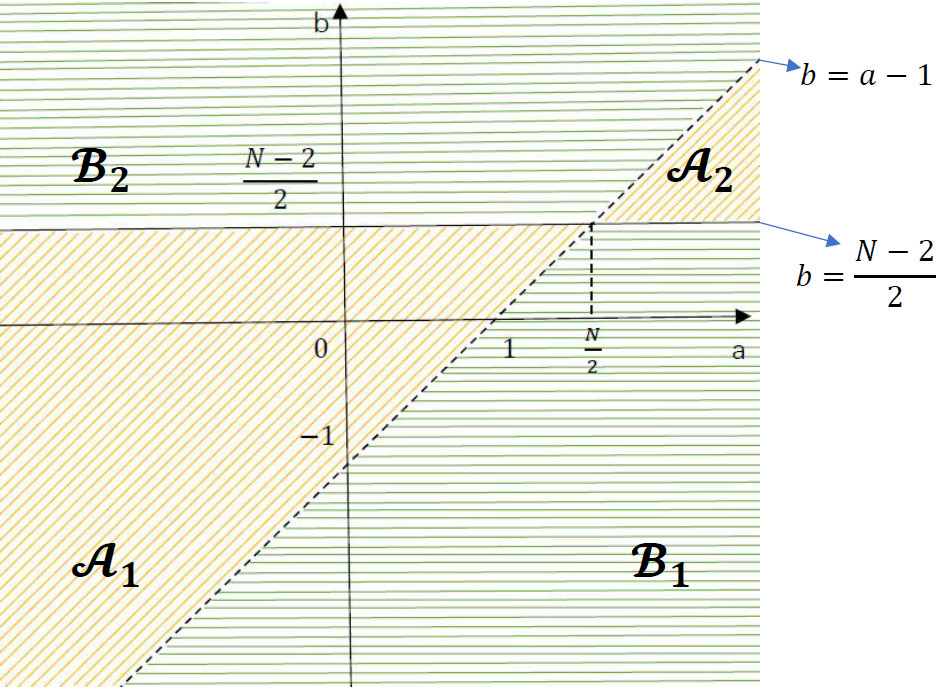}
\caption{The regions for $(a, b)\in \R^2$.}
\label{fig:1}
\hfill
\end{figure}

\section{Sharp constants and optimizers for
the anisotropic Heisenberg Uncertainty Principle}\label{s4}
This section is devoted to the sharp anisotropic Heisenberg uncertainty principle associated
with the Minkowski functional. The proof combines the anisotropic
Hardy-type identity in Section \ref{s3} with the anisotropic P\'olya-Szeg\"{o} principle and the weighted anisotropic
Hardy-Littlewood inequality.
Let $$u\in \mathcal{W}:=W^{1,2}(\RN)\cap\bigg\{u:\int_{\RN}\hK{x}^2|u|^2dx<\infty\bigg\}.$$
\begin{theorem}\label{t2}
Assume that $K\in\Ksc$ and $N \geq 2$. For any function $u \in \mathcal{W}$,
we have
\begin{itemize}
\item[(i)]
\begin{align}\label{f12}
\int_{\RN}\|-\nabla u\|_{K^*}^2dx+\int_{\RN}\hK{x}^2 u^2dx \geq N \int_{\RN} u^2dx,
\end{align}
where the constant $N$ is sharp and is attained by 
$u=\kappa e^{-\frac{1}{2}\hK{x}^2}$
with $\kappa\ge0$.

\item[(ii)]
\begin{align}\label{f71}
\bigg(\int_{\RN}\|-\nabla u\|_{K^*}^2dx\bigg)^{\frac{1}{2}}
\bigg(\int_{\RN}\hK{x}^2 u^2dx\bigg)^{\frac{1}{2}}
\geq \frac{N}{2}\int_{\RN} u^2dx,
\end{align}
where the constant $\frac{N}{2}$ is sharp  and is attained by
$
u(x)=\beta e^{-\frac{1}{2\lambda^2}\hK{x}^2}
$
with $\beta\ge0$ and
$
\lambda>0.
$ 
\end{itemize}
\end{theorem}
\begin{proof}
($i$)
By the anisotropic P\'olya-Szeg\"{o} principle \eqref{f54}, \eqref{E1-1}, using the
equimeasurability of $u^2$ and $(u^K)^2$ and applying Proposition \ref{p1} $(ii)$ to $|u|$ with
$p=2$ and $v(r)=r^2$ due to $u^K=(|u|)^K$, one has
\begin{align}
\int_{\RN} \|-\nabla u(x)\|_{K^*}^2dx
&\ge
\int_{\RN} \|-\nabla u^K(x)\|_{K^*}^2dx\label{f32}\\
&=
\int_{\RN} |\mathcal R_K \big(u^K\big)|^2dx\notag\\
&\ge
N\int_{\RN} \big(u^K(x)\big)^2dx
-
\int_{\RN} \hK{x}^2 \big(u^K(x)\big)^2dx\label{f33}\\
&\ge
N\int_{\RN} u(x)^2dx
-
\int_{\RN} \hK{x}^2 u(x)^2dx\label{f34}.
\end{align}
This proves inequality \eqref{f12}.

Assume  $u(x)=\kappa e^{-\frac12\hK{x}^2}$ for $\kappa\ge0$. Then $u$ is radial, and hence $u=u^K$ up to a null set. Moreover,
$
\mathcal{R}_K(u)+\hK{x} u=0.
$
Thus, the equality holds in \eqref{f32} \eqref{f33} and \eqref{f34}.
Hence the equality holds in \eqref{f12}, and the constant $N$ is sharp.

\noindent
($ii$) By \eqref{f54}, \eqref{f35}, \eqref{c3} and $u^2$ and
$(u^K)^2$ are equimeasurable, one has
\begin{align}
&\bigg(\int_{\RN} \|-\nabla u(x)\|_{K^*}^2dx\bigg)^{\frac{1}{2}}\bigg(\int_{\RN}\hK{x}^2u^2dx\bigg)^{\frac{1}{2}}\notag\\
&\ge
\bigg(\int_{\RN} \|-\nabla u^K(x)\|_{K^*}^2dx\bigg)^{\frac{1}{2}}\bigg(\int_{\RN}\hK{x}^2u^2dx\bigg)^{\frac{1}{2}}\label{f40}\\
&\ge
\bigg(\int_{\RN} \big|\mathcal{R}_K(u^K)\big|^2dx \bigg)^{\frac{1}{2}}\bigg(\int_{\RN}\hK{x}^2\big(u^K\big)^2dx\bigg)^{\frac{1}{2}}\label{f41}\\
&\ge\frac{N}{2}\int_{\RN}\big(u^K\big)^2dx\label{f42}\\
&=\frac{N}{2}\int_{\RN}u^2dx.\notag
\end{align}
This proves inequality \eqref{f71}.

When $\beta=0$, the equality in \eqref{f71} is immediate. Hence, it remains to
consider the case $\beta>0$.
Assume that
$$
u(x)=\beta \exp\!\left(-\frac{1}{2\lambda^2}\hK{x}^2\right)
$$
for some constants $\beta>0$ and $\lambda>0$. 
Then $u$ is radial, and hence $u=u^K$ up to a null set. Moreover,
$
\mathcal R_K(u)
=
-\frac{1}{\lambda^2}\hK{x}u(x).
$
Hence,
$$
\|\mathcal R_K(u)\|
=
\frac{1}{\lambda^2}\big\|\hK{x}u\big\|,
$$
and then,
$$
\lambda
=
\bigg(
\frac{\big\|\hK{x}u\big\|}{\|\mathcal R_K(u)\|}
\bigg)^{\frac{1}{2}}.
$$
It follows that
\begin{align*}
\frac{\big\|\hK{x}u^K\big\|^{\frac{1}{2}}\mathcal{R}_K(u^K)}{\|\mathcal{R}_K(u^K)\|^{\frac{1}{2}}}+\frac{\|\mathcal{R}_K(u^K)\|^{\frac{1}{2}}\hK{x}u^K}{\big\|\hK{x}u^K\big\|^{\frac{1}{2}}}=
-\frac{1}{\lambda}\hK{x}u^K
+
\frac{1}{\lambda}\hK{x}u^K
=0.
\end{align*}
Therefore, the equality holds in \eqref{f40}, \eqref{f41} and \eqref{f42}. 
Hence, the equality holds in inequality \eqref{f71}, and the constant $\frac{N}{2}$ is sharp. 
\end{proof}

\section{Sharp constants and optimizers for
the anisotropic CKN inequality}\label{s5}
We now turn to the sharp anisotropic CKN inequalities. The anisotropic CKN identities \eqref{f55}, \eqref{f75}, \eqref{f56} and \eqref{f61} with remainders obtained in Section \ref{s3}
give the following two constants,
$$
C_1(N,a,b)=\bigg|\frac{N-a-b-1}{2}\bigg|
\ \ \mathrm{and}\ \
C_2(N,a,b)=\bigg|\frac{N-3b+a-3}{2}\bigg|.
$$
 The comparison of $C_1$ and $C_2$ determines
which identity yields the stronger anisotropic CKN inequality in each part of the $(a,b)$-parameter plane.
Indeed,
$
C_1(N,a,b)^2-C_2(N,a,b)^2
=
(b-a+1)(N-2b-2).
$
Hence, $C_1(N,a,b)> C_2(N,a,b)$ in the region where
$(b-a+1)(N-2b-2)>0$, $C_2(N,a,b)>C_1(N,a,b)$ in the region where
$(b-a+1)(N-2b-2)<0$, and $C_2(N,a,b)=C_1(N,a,b)$ if $b=a-1$ or $b=\frac{N-2}{2}$. Here the line $b=a-1$ 
is the critical line where the anisotropic CKN equality equations reduce to the Hardy
case and where nonzero extremal functions are not attained in a completion space, and the line 
$b=\frac{N-2}{2}$ only corresponds to the coincidence
of the two constants. Accordingly, we use the following disjoint decomposition of the noncritical part of the $(a,b)$-parameter plane; (see Figure \ref{fig:1}):
\begin{align*}
\mathcal{A}_1:&=\Big\{(a,b)\in\R^2\Big|\ b>a-1,\ b\leq\frac{N-2}{2}\Big\};\\
\mathcal{A}_2:&=\Big\{(a,b)\in\R^2\Big|\ b<a-1,\ b\geq\frac{N-2}{2}\Big\};\\
\mathcal{B}_1:&=\Big\{(a,b)\in\R^2\Big|\ b<a-1,\ b<\frac{N-2}{2}\Big\};\\
\mathcal{B}_2:&=\Big\{(a,b)\in\R^2\Big|\ b>a-1,\ b>\frac{N-2}{2}\Big\}.
\end{align*}
Let $\mathcal{A}:=\mathcal{A}_1\cup \mathcal{A}_2$ and $\mathcal{B}:=\mathcal{B}_1\cup \mathcal{B}_2$.
Note that $\mathcal{A}\cup\mathcal{B}=\big\{(a,b)\in\R^2\big|\ a\neq b+1\big\}.$
The sharpness will be proved below by deriving the equality equations from the
vanishing of the corresponding remainders and by showing that the extremal functions
belong to the completion space $\mathcal C_{K,a,b}$,
the completion of
$C_c^\infty(\RN\setminus\{o\})$ under the norm
$$
\|u\|_{\mathcal C_{K,a,b}}
=
\bigg(
\int_{\RN}\frac{|\mathcal R_K(u)|^2}{\hK{x}^{2b}}dx
+
\int_{\RN}\frac{u^2}{\hK{x}^{2a}}dx
\bigg)^{\frac12}.
$$

\begin{theorem}\label{t3}
Let $K\in\Ksc$ and $a,b\in\R$. The anisotropic CKN inequality (\ref{t-1}) holds for all
$u\in C_c^\infty(\RN\setminus\{o\})$, with the zero function
giving only the trivial equality case. Its sharp constants are as follows.
\begin{itemize}
\item[(i)] In the region $\mathcal{A}$, the sharp constant is
 $\widetilde{C}(N,a,b)=\big|\frac{N-a-b-1}{2}\big|$ and it is achieved in $\mathcal C_{K,a,b}$ by the nonzero functions
$$
u(x)=\Phi\big(\sigma_K(x)\big) e^{\frac{t}{b+1-a}\hK{x}^{b+1-a}},
$$
where $t<0$ in $\mathcal{A}_1$, $t>0$ in $\mathcal{A}_2$, and $0\not\equiv\Phi\in L^2(\partial K,d\mu_K)$.

\item[(ii)] In the region $\mathcal{B}$, the sharp constant is
$\widetilde{C}(N,a,b)=\big|\frac{N-3b+a-3}{2}\big|$ and it is achieved in $\mathcal C_{K,a,b}$ by the nonzero functions
$$
u(x)=\Phi\big(\sigma_K(x)\big)\hK{x}^{2b+2-N}e^{\frac{t}{b+1-a}\hK{x}^{b+1-a}},
$$
where $t>0$ in $\mathcal{B}_1$, $t<0$ in $\mathcal{B}_2$, and $0\not\equiv\Phi\in L^2(\partial K,d\mu_K)$.

\item[(iii)] When $a=b+1$, the sharp constant is $\widetilde{C}(N,b+1,b)=\frac{|N-2(b+1)|}{2}$, and cannot be achieved by the nonzero functions in $\mathcal C_{K,a,b}$.
\end{itemize}
\end{theorem}

\begin{proof}
Let $r=\hK{x}$. The case $u\equiv 0$ is trivial. Without loss of generality, we assume that
$u\not\equiv 0$.
We first prove the inequality and derive the corresponding equality equations from
the remainder identities established in Section \ref{s3}. The sharpness will follow once the corresponding equality equations are solved and
the extremal functions are shown to belong to the completion space
$\mathcal C_{K,a,b}$.

\medskip
\noindent\emph{The region $\mathcal A$.}
Let $(a,b)\in\mathcal A$. Then
$C_1(N,a,b)\ge C_2(N,a,b)$ for $(a,b)\in\mathcal A$. Thus anisotropic CKN identities \eqref{f55} and \eqref{f75} give the following anisotropic CKN inequality
\begin{align*}
\Big(\int_{\RN}\frac{\big|\mathcal{R}_K(u)\big|^2}{\hK{x}^{2b}}dx\Big)^{\frac{1}{2}}\Big(\int_{\RN}\frac{u^2}{\hK{x}^{2a}}dx\Big)^{\frac{1}{2}}\geq \Big|\frac{N-a-b-1}{2}\Big|\int_{\RN}\frac{u^2}{\hK{x}^{a+b+1}}dx.
\end{align*} 
The equality holds if and only if the corresponding nonnegative
remainder in \eqref{f55} or \eqref{f75} vanishes, namely
$$
\mathrm{sgn}(N-a-b-1)
\frac{
\big(\int_{\RN}u^2r^{-2a}dx\big)^{\frac14}
}{
\big(\int_{\RN}|\mathcal R_K(u)|^2r^{-2b}dx\big)^{\frac14}
}
\frac{\mathcal R_K(u)}{r^b}
+
\frac{
\big(\int_{\RN}|\mathcal R_K(u)|^2r^{-2b}dx\big)^{\frac14}
}{
\big(\int_{\RN}u^2r^{-2a}dx\big)^{\frac14}
}
\frac{u}{r^a}
=0
$$
almost everywhere, where $\mathrm{sgn}(t)$ is the sign of $t\in \R\setminus{0}$ and $\mathrm{sgn}(0)=0$. Equivalently,
$
\mathcal R_K(u)=t r^{b-a}u,
$
 where
$$
t
=
-\mathrm{sgn}(N-a-b-1)
\frac{
\big(\int_{\RN}|\mathcal R_K(u)|^2r^{-2b}dx\big)^{\frac12}
}{
\big(\int_{\RN}u^2r^{-2a}dx\big)^{\frac12}
}.
$$
Since $\mathcal R_K(u)=\partial_r u$, then
$
\partial_r u=t r^{b-a}u 
$
and hence, 
\begin{align*}
u(x)=\Phi(\sigma_K(x))
\exp\!\bigg(\frac{t\hK{x}^{b+1-a}}{b+1-a}\bigg),
\end{align*}
where $t<0$ in $\mathcal A_1$ and $t>0$ in $\mathcal A_2$. Therefore, once the above extremal functions are shown to belong to $\mathcal C_{K,a,b}$,
the constant $$C_1(N,a,b)=\frac{|N-a-b-1|}{2}
$$ is sharp in $\mathcal A$.

\medskip
\noindent\emph{The region $\mathcal B$.}
Let $(a,b)\in\mathcal B$. Then $C_2(N,a,b)>C_1(N,a,b).$
If $(a,b)\in\mathcal B_1$, we use
the identity \eqref{f56}; if $(a,b)\in\mathcal B_2$, we use the
identity \eqref{f61}. In both cases, these identities yield
\begin{align}\label{f58}
\bigg(\int_{\RN}\frac{|\mathcal R_K(u)|^2}{\hK{x}^{2b}}dx\bigg)^{\frac12}
\bigg(\int_{\RN}\frac{u^2}{\hK{x}^{2a}}dx\bigg)^{\frac12}
\ge
\frac{|N-3b+a-3|}{2}
\int_{\RN}\frac{u^2}{\hK{x}^{a+b+1}}dx.
\end{align}
The equality in \eqref{f58} is obtained from the vanishing of the nonnegative remainder
in the corresponding identity. In both subregions, one has
$
\mathcal R_K(u)
=
t r^{b-a}u-(N-2b-2)r^{-1}u
$
for some constant $t\in\mathbb R$. Equivalently,
$
\partial_r u
=
t r^{b-a}u
-
(N-2b-2)r^{-1}u .
$
Hence,
\begin{align*}
u(x)=
\Phi(\sigma_K(x))\hK{x}^{2b+2-N}
\exp\!\bigg(\frac{t\hK{x}^{b+1-a}}{b+1-a}\bigg),
\end{align*}
where $t>0$ in $\mathcal B_1$ and $t<0$ in $\mathcal B_2$. Hence, once the above extremal functions are checked to belong to $\mathcal C_{K,a,b}$,
the constant $$C_2(N,a,b)=\frac{|N-3b+a-3|}{2}
$$ is sharp in $\mathcal B$.

\medskip
\noindent\emph{The critical line $a=b+1$.}
Assume $a=b+1$. Then
$$
\frac{u(x)^2}{\|x\|_K^{2a}}=\frac{u(x)^2}{\|x\|_K^{a+b+1}}=\frac{u(x)^2}{\|x\|_K^{2b+2}}
$$
and hence the anisotropic CKN inequality reduces to a weighted anisotropic Hardy-type inequality.
Taking
$
\lambda=2b
$
in the weighted anisotropic $L^2$-Hardy identity \eqref{f60}, one has the following anisotropic Hardy identity
\begin{align}\label{f67}
\!\!\!\!\!\!\int_{\RN}
\frac{|\mathcal R_K(u)|^2}{\|x\|_K^{2b}}dx
\!=\!\!
\bigg(\frac{N-2b-2}{2}\bigg)^2
\!\!\int_{\RN}
\frac{u^2}{\|x\|_K^{2a}}dx+\!\!
\int_{\RN}
\bigg|
\frac{\mathcal R_K(u)}{\|x\|_K^b}
+\!
\frac{N-2b-2}{2}
\frac{u}{\|x\|_K^{b+1}}
\bigg|^2dx.
\end{align}
It follows that
\begin{align}\label{f69}
\int_{\RN}
\frac{|\mathcal R_K(u)|^2}{\|x\|_K^{2b}}dx
\ge
\bigg(\frac{N-2b-2}{2}\bigg)^2
\int_{\RN}
\frac{u^2}{\|x\|_K^{2a}}dx.
\end{align}
Hence,
\begin{align}\label{f66}
\widetilde C(N,b+1,b)\ge \frac{|N-2b-2|}{2}.
\end{align}
If the equality in \eqref{f69} was attained by a nonzero function
$u\in\mathcal C_{K,b+1,b}$, then the Hardy remainder in \eqref{f67} would vanish. Hence, one has
$$
\mathcal R_K(u)+\frac{N-2b-2}{2}\frac{u}{\hK{x}}=0,
$$
and therefore,
$
u(x)=\Phi(\sigma_K(x))\hK{x}^{-\frac{N-2b-2}{2}}.
$
However,
$$
\int_{\RN}\frac{u^2}{\hK{x}^{2b+2}}dx
=
\int_{\partial K}|\Phi(\sigma)|^2d\mu_K(\sigma)
\int_0^\infty \frac{ds}{s},
$$
which is finite only when $\Phi=0$  a.e. on $\partial K$. Thus, no nonzero function in
$\mathcal C_{K,b+1,b}$ attains the equality in \eqref{f69}.

Let $0<\varepsilon<1$. Define
$\varphi_\varepsilon\in C_c^\infty(0,\infty)$ such that
$0\le\varphi_\varepsilon\le1$, 
$$\varphi_\varepsilon(r)=
\begin{cases}
1, & \varepsilon\le r\le \varepsilon^{-1},\\
0, & 0<r\le \varepsilon^2 \ \text{or}\ r\ge \varepsilon^{-2},
\end{cases}$$
and
$$
|\varphi_\varepsilon'(r)|
\le
\frac{M}{r\log(1/\varepsilon)}
\quad\text{on }
(\varepsilon^2,\varepsilon)\cup(\varepsilon^{-1},\varepsilon^{-2})
$$
where the constant $M>0$ is independent of $\varepsilon$.
Define
$
u_\varepsilon(x)=\varphi_\varepsilon(\hK{x})\hK{x}^{-\frac{N-2b-2}{2}}.
$
Then
$$
\frac{\mathcal R_K(u_\varepsilon)}{\hK{x}^{b}}
+
\frac{N-2b-2}{2}\frac{u_\varepsilon}{\hK{x}^{b+1}}
=
\varphi_\varepsilon'(\hK{x})\hK{x}^{-\frac{N-2b-2}{2}-b}.
$$
By the anisotropic polar formula, one has
$$
\int_{\RN}\frac{u_\varepsilon^2}{\hK{x}^{2b+2}}dx
=\mu_K(\partial K)\int_0^\infty \varphi_\varepsilon(r)^2\frac{dr}{r}\ge
\mu_K(\partial K)\int_\varepsilon^{\varepsilon^{-1}} \frac{dr}{r}
= 2\mu_K(\partial K)\log(1/\varepsilon),
$$
\begin{align*}
\int_{\RN}
\Bigg|
\frac{\mathcal R_K(u_\varepsilon)}{\hK{x}^{b}}
+
\frac{N-2b-2}{2}\frac{u_\varepsilon}{\hK{x}^{b+1}}
\Bigg|^2dx
&=
\mu_K(\partial K)\int_0^\infty |\varphi_\varepsilon'(r)|^2rdr\\
&\le\frac{\mu_K(\partial K)M^2}{\log(1/\varepsilon)^2}\Big(\int_{\varepsilon^2}^\varepsilon \frac{dr}{r}+\int_{\varepsilon^{-1}}^{\varepsilon^{-2}} \frac{dr}{r}\Big)\\
&=
\frac{2\mu_K(\partial K)M^2}{\log(1/\varepsilon)}.
\end{align*}
Therefore,
$$
\frac{
\int_{\RN}
\Big|
\frac{\mathcal R_K(u_\varepsilon)}{\hK{x}^{b}}
+
\frac{N-2b-2}{2}\frac{u_\varepsilon}{\hK{x}^{b+1}}
\Big|^2dx
}{
\int_{\RN}\frac{u_\varepsilon^2}{\hK{x}^{2b+2}}dx
}
\to0
\quad \text{as }\varepsilon\to0.
$$
Using the anisotropic Hardy identity \eqref{f67}, one gets
$$
\frac{
\Big(\int_{\RN}\frac{|\mathcal R_K(u_\varepsilon)|^2}{\hK{x}^{2b}}dx\Big)^{\frac12}
\Big(\int_{\RN}\frac{u_\varepsilon^2}{\hK{x}^{2b+2}}dx\Big)^{\frac12}
}{
\int_{\RN}\frac{u_\varepsilon^2}{\hK{x}^{2b+2}}dx
}
\to \frac{|N-2b-2|}{2}\quad \text{as }\varepsilon\to0.
$$
Thus, 
\begin{align}\label{f68}
\widetilde C(N,b+1,b)\le \frac{|N-2b-2|}{2}.
\end{align}
Together with \eqref{f66} and \eqref{f68}, one gets the sharp constant
$$
\widetilde C(N,b+1,b)
=
\frac{|N-2(b+1)|}{2}.
$$

\medskip
\noindent\emph{Attainability in the completion space.}
It remains to prove that the extremal functions obtained above belong to
$\mathcal C_{K,a,b}$. 
Assume that $(a,b)\in\mathcal A\cup\mathcal B$. Let
$
u(x)=\Phi(\sigma_K(x))F(\hK{x}),
\
\Phi\in L^2(\partial K,d\mu_K),
$
where $F:(0,\infty)\to\mathbb R$ is defined by
$$
F(r):=
\begin{cases}
\exp\big(\frac{tr^{b+1-a}}{b+1-a}\big),
& \text{if } (a,b)\in\mathcal A,\\[2.5mm]
r^{2b+2-N}
\exp\big(\frac{tr^{b+1-a}}{b+1-a}\big),
& \text{if } (a,b)\in\mathcal B.
\end{cases}
$$
The sign of $t$ in the four regions $\mathcal A_1,\mathcal A_2,\mathcal B_1$
and $\mathcal B_2$ gives
$$
\int_0^\infty |F(r)|^2r^{N-1-2a}dr
+
\int_0^\infty |F'(r)|^2r^{N-1-2b}dr
<\infty.
$$
Since $d\mu_K$ is a finite Borel measure on the compact set $\partial K$ and
the restrictions of $C^1$ functions are dense in $L^2(\partial K,d\mu_K)$,  there exists a sequence $\Phi_j:\partial K\to\mathbb R$ such that
$
x\mapsto \Phi_j(\sigma_K(x))\in C^1(\RN\setminus\{o\})
$ 
and
$
\Phi_j\to \Phi$
in $L^2(\partial K,d\mu_K)
$
as $j\to\infty$.
Let
$
u_j(x):=\Phi_j\big(\sigma_K(x)\big)F(\hK{x}).
$
Using the anisotropic polar formula and
$
\mathcal R_K\big(\Phi_j\big(\sigma_K(x)\big)F(\hK{x})\big)
=
\Phi_j\big(\sigma_K(x)\big)F'(\hK{x})
$, one has
\begin{align*}
\|u_j-u\|_{\mathcal C_{K,a,b}}^2
&=
\int_{\RN}
\frac{|u_j-u|^2}{\hK{x}^{2a}}dx
+
\int_{\RN}
\frac{|\mathcal R_K(u_j-u)|^2}{\hK{x}^{2b}}dx  \\
&=
\|\Phi_j-\Phi\|_{L^2(\partial K,d\mu_K)}^2
\bigg[
\int_0^\infty |F(r)|^2r^{N-1-2a}dr
+
\int_0^\infty |F'(r)|^2r^{N-1-2b}dr
\bigg].
\end{align*}
It follows that
$
u_j\to u$
 as $j\to\infty$ in $\mathcal C_{K,a,b}
$.

For each fixed $j$, define
$
v_{j,\varepsilon}(x):=\varphi_\varepsilon(\hK{x})u_j(x)
$ for $x\in\RN$.
We claim that, for each fixed $j$,
$
\|v_{j,\varepsilon}-u_j\|_{\mathcal C_{K,a,b}}\to0$ as $\varepsilon\to0.
$
In fact, one has
$
v_{j,\varepsilon}(x)-u_j(x)=\big(\varphi_\varepsilon(\hK{x})-1\big)u_j(x).
$
Since $0\le\varphi_\varepsilon\le1$ and $\varphi_\varepsilon\to1$ a.e. as $\varepsilon\to0$, then the dominated
convergence theorem gives
$$
\int_{\RN}
\frac{|v_{j,\varepsilon}-u_j|^2}{\hK{x}^{2a}}dx
=
\int_{\RN}
|1-\varphi_\varepsilon|^2
\frac{|u_j|^2}{\hK{x}^{2a}}dx
\to0.
$$
Since
$
\mathcal R_K(v_{j,\varepsilon}-u_j)
=
(\varphi_\varepsilon-1)\mathcal R_K(u_j)
+
u_j\mathcal R_K(\varphi_\varepsilon)
$ and $\mathcal R_K\big(\varphi_\varepsilon(\|x\|_K)\big)
=
\varphi_\varepsilon'(\|x\|_K)
\sigma_K(x)\cdot\nabla\|x\|_K
=
\varphi_\varepsilon'(\|x\|_K)$,
then 
\begin{align*}
\int_{\RN}
\frac{|\mathcal R_K(v_{j,\varepsilon}-u_j)|^2}{\hK{x}^{2b}}dx\le
2\int_{\RN}
|1-\varphi_\varepsilon|^2
\frac{|\mathcal R_K(u_j)|^2}{\hK{x}^{2b}}dx+
2\int_{\RN}
\frac{|u_j|^2|\varphi_\varepsilon'(\|x\|_K)|^2}{\hK{x}^{2b}}dx.
\end{align*}
The first term tends to zero by the dominated
convergence theorem. It remains to estimate
the second term. By the definition of $\varphi_\varepsilon$ above, one has
\begin{align*}
\int_{\RN}
\frac{|u_j|^2|\varphi_\varepsilon'(\hK{x})|^2}
{\hK{x}^{2b}}dx
&\le
\frac{M}{\log^2(1/\varepsilon)}
\int_{\{\varepsilon^2<\hK{x}<\varepsilon\}
\cup
\{\varepsilon^{-1}<\hK{x}<\varepsilon^{-2}\}}
\frac{|u_j|^2}{\hK{x}^{2b+2}}dx \\
&=
\frac{M\|\Phi_j\|_{L^2(\partial K,d\mu_K)}^2}{\log^2(1/\varepsilon)}
\int_{(\varepsilon^2,\varepsilon)\cup(\varepsilon^{-1},\varepsilon^{-2})}
|F(r)|^2r^{N-2b-3}dr.
\end{align*}
We claim that
$$
\int_{(\varepsilon^2,\varepsilon)\cup(\varepsilon^{-1},\varepsilon^{-2})}
|F(r)|^2r^{N-2b-3}dr
\le
M'\log(1/\varepsilon)
$$ for some constant $M'>0.$
Indeed, 
if $(a,b)\in\mathcal A_1$, then $b+1-a>0$, $N-3-2b\ge -1$ and $t<0$. Hence, one has
\begin{align*}
\int_{(\varepsilon^2,\varepsilon)\cup(\varepsilon^{-1},\varepsilon^{-2})}
|F(r)|^2r^{N-2b-3}dr
&=
\int_{(\varepsilon^2,\varepsilon)\cup(\varepsilon^{-1},\varepsilon^{-2})}
r^{N-2b-3}
\exp\Big(\frac{2t}{b+1-a}r^{b+1-a}\Big)dr\\
&\le
\int_{\varepsilon^2}^{\varepsilon}r^{N-2b-3}dr+\int_1^\infty
r^{N-2b-3}
\exp\Big(\frac{2t}{b+1-a}r^{b+1-a}\Big)dr\\
&\le M_{\mathcal A_1}\log(1/\varepsilon) \ \mathrm{for\ some\ constant\ }M_{\mathcal A_1}>0.
\end{align*}
If $(a,b)\in\mathcal A_2$, then $b+1-a<0$, $N-3-2b\le -1$ and $t>0$. Hence, one has
\begin{align*}
\int_{(\varepsilon^2,\varepsilon)\cup(\varepsilon^{-1},\varepsilon^{-2})}
|F(r)|^2r^{N-2b-3}dr
&=
\int_{(\varepsilon^2,\varepsilon)\cup(\varepsilon^{-1},\varepsilon^{-2})}
r^{N-2b-3}
\exp\Big(-\frac{2t}{a-b-1}r^{-(a-b-1)}\Big)dr\\
&\le\int_0^1
r^{N-2b-3}
\exp\Big(-\frac{2t}{a-b-1}r^{-(a-b-1)}\Big)dr+
\int_{\varepsilon^{-1}}^{\varepsilon^{-2}}r^{N-2b-3}dr
\\
&\le M_{\mathcal A_2}\log(1/\varepsilon)\ \mathrm{for\ some\ constant\ }M_{\mathcal A_2}>0.
\end{align*}
If $(a,b)\in\mathcal B_1$, then $b+1-a<0$, $2b+1-N<-1$ and $t>0$. Hence, one has
\begin{align*}
\int_{(\varepsilon^2,\varepsilon)\cup(\varepsilon^{-1},\varepsilon^{-2})}
|F(r)|^2r^{N-2b-3}dr
&=
\int_{(\varepsilon^2,\varepsilon)\cup(\varepsilon^{-1},\varepsilon^{-2})}
r^{2b+1-N}
\exp\Big(-\frac{2t}{a-b-1}r^{-(a-b-1)}\Big)dr\\
&\le\int_0^1
r^{2b+1-N}
\exp\Big(-\frac{2t}{a-b-1}r^{-(a-b-1)}\Big)dr+
\int_{\varepsilon^{-1}}^{\varepsilon^{-2}}r^{2b+1-N}dr
\\
&\le M_{\mathcal B_1}\log(1/\varepsilon)\ \mathrm{for\ some\ constant\ }M_{\mathcal B_1}>0.
\end{align*}
If $(a,b)\in\mathcal B_2$, then $b+1-a>0$, $2b+1-N>-1$ and $t<0$. Hence, one has
\begin{align*}
\int_{(\varepsilon^2,\varepsilon)\cup(\varepsilon^{-1},\varepsilon^{-2})}
|F(r)|^2r^{N-2b-3}dr
&=
\int_{(\varepsilon^2,\varepsilon)\cup(\varepsilon^{-1},\varepsilon^{-2})}
r^{2b+1-N}
\exp\Big(\frac{2t}{b+1-a}r^{b+1-a}\Big)dr\\
&\le\int_0^1 r^{2b+1-N}dr+
\int_1^\infty
r^{2b+1-N}
\exp\Big(\frac{2t}{b+1-a}r^{b+1-a}\Big)dr
\\
&\le M_{\mathcal B_2}\log(1/\varepsilon)\ \mathrm{for\ some\ constant\ }M_{\mathcal B_2}>0.
\end{align*}
Consequently,
$$
\int_{\RN}
\frac{|u_j|^2|\varphi_\varepsilon'(\hK{x})|^2}{\hK{x}^{2b}}dx
\le
\frac{M''}{\log(1/\varepsilon)}
\to0\ \mathrm{as}\ \varepsilon\to0
$$ for\ some\ constant $M''>0$.
Therefore,
$$
\|v_{j,\varepsilon}-u_j\|_{\mathcal C_{K,a,b}}\to0
\quad \text{as }\varepsilon\to0.
$$

For fixed $j$ and $\varepsilon$, let
$
\Omega_\varepsilon
=
\big\{x\in\RN:
\frac{\varepsilon^2}{2}<\hK{x}<\frac{2}{\varepsilon^{2}}
\big\}.
$
Then $v_{j,\varepsilon}$ has compact support in $\Omega_\varepsilon$,
and $\Omega_\varepsilon$ is chosen away from the origin.
Since $v_{j,\varepsilon}\in W^{1,2}(\Omega_\varepsilon)$ has compact support
in $\Omega_\varepsilon$, it follows that
$
v_{j,\varepsilon}\in W^{1,2}_c(\Omega_\varepsilon)
$, where
$W^{1,2}_c(\Omega_\varepsilon)$ denotes the closure of
$C_c^\infty(\Omega_\varepsilon)$ in the norm $\|\cdot\|_{W^{1,2}(\Omega_\varepsilon)}$.
By the density theorem, there exists a sequence
$
w_{j,\varepsilon,n}\in C_c^\infty(\Omega_\varepsilon)
\subset C_c^\infty(\RN\setminus\{o\})
$
such that
$
w_{j,\varepsilon,n}\to v_{j,\varepsilon}
$
as $n\to\infty$ in $W^{1,2}(\Omega_\varepsilon)$. As $\sigma_K$ is bounded on $\Omega_\varepsilon$, then
$
\mathcal R_K(w_{j,\varepsilon,n}-v_{j,\varepsilon})\to0$  as $n\to\infty$ in $L^2(\Omega_\varepsilon)$.
Note that  $\hK{x}^{-2a}$ and $\hK{x}^{-2b}$ are bounded on
$\Omega_\varepsilon$. Therefore,
$$
\|w_{j,\varepsilon,n}-v_{j,\varepsilon}\|_{\mathcal C_{K,a,b}}\to0
\quad \text{as }n\to\infty.
$$

Now for fixed  $j$, we choose $\varepsilon_j>0$ so small that
$$
\|v_{j,\varepsilon_j}-u_j\|_{\mathcal C_{K,a,b}}<\frac1{2j},
$$
and then we choose $n_j$ so large that
$$
\|w_{j,\varepsilon_j,n_j}-v_{j,\varepsilon_j}\|_{\mathcal C_{K,a,b}}
<\frac1{2j}.
$$
It follows that
\begin{align*}
\|w_{j,\varepsilon_j,n_j}-u\|_{\mathcal C_{K,a,b}}
\!\le\!
\|w_{j,\varepsilon_j,n_j}-v_{j,\varepsilon_j}\|_{\mathcal C_{K,a,b}}
\!+\!
\|v_{j,\varepsilon_j}-u_j\|_{\mathcal C_{K,a,b}}
\!+\!
\|u_j\!-\!u\|_{\mathcal C_{K,a,b}}\! \le\!
\frac1j\!+\!\|u_j\!-\!u\|_{\mathcal C_{K,a,b}}.
\end{align*}
Since $u_j\to u$ as $j\to\infty$ in $\mathcal C_{K,a,b}$, one obtains
$$
w_{j,\varepsilon_j,n_j}\to u
\ \ \text{as }j\to\infty
$$ in $\mathcal C_{K,a,b}$.
 Hence,
$
u\in\mathcal C_{K,a,b},
$
and then
the corresponding
constants $C_1$ in $\mathcal A$ and $C_2$ in $\mathcal B$ are sharp.
This completes the proof.
\end{proof}
As a direct consequence of Theorem \ref{t3} and the anisotropic Heisenberg-type identity \eqref{E1-1}, we can obtain the following sharp anisotropic Heisenberg-type inequalities.
\begin{corollary}
Let $K\in\Ksc$ and $N\ge2$. Then, for 
$u\in C_c^\infty(\mathbb R^N\setminus\{o\})$, one has
\begin{align}\label{f77}
\int_{\mathbb R^N}|\mathcal R_K(u)|^2dx
+
\int_{\mathbb R^N}\hK{x}^2u^2dx
\ge
N\int_{\mathbb R^N}u^2dx,
\end{align}
\begin{align}\label{f78}
\bigg(\int_{\mathbb R^N}|\mathcal R_K(u)|^2dx\bigg)^{\frac12}
\bigg(\int_{\mathbb R^N}\hK{x}^2u^2dx\bigg)^{\frac12}
\ge
\frac N2
\int_{\mathbb R^N}u^2dx.  
\end{align}
The sharp constants are $N$ and $\frac{N}{2}$, respectively, and they are achieved in $\mathcal C_{K,-1,0}$.
The equality in \eqref{f77} is attained by
$$
u(x)=\Phi(\sigma_K(x))
\exp\bigg(-\frac12\hK{x}^2\bigg),
$$
while the equality in \eqref{f78} is attained by
$$
u(x)=\Phi(\sigma_K(x))
\exp\bigg(-\frac{1}{2\lambda^2}\hK{x}^2\bigg),
\quad \lambda>0,
$$
where $\Phi\in L^2(\partial K,d\mu_K)$.
\end{corollary}

\begin{proof}
From anisotropic Heisenberg-type identity \eqref{E1-1}, one has inequality \eqref{f77}. The equality in \eqref{f77} holds when the remainder in \eqref{E1-1} vanishes,
that is,
$
\mathcal R_K(u)+\hK{x}u=0,
$
it follows that
$$
u(x)=\Phi\big(\sigma_K(x)\big)
\exp\bigg(-\frac12\hK{x}^2\bigg).
$$
 Inequality \eqref{f78} and its equality cases follow from Theorem \ref{t3} by taking
$a=-1$ and $b=0$. 
The above extremal functions in $\mathcal C_{K,-1,0}$ also
follow from Theorem \ref{t3} with $a=-1$ and $b=0$.
\end{proof}

Theorem \ref{t3} gives the sharp anisotropic CKN inequality in terms of the anisotropic
radial derivative $\mathcal{R}_K(u)$. We next derive a corresponding estimate involving
the anisotropic gradient. The basic observation is that the identity $\hK{x}=\|-x\|_K$ does not necessarily hold for $K\in\Ksc\setminus\mathcal{K}_{(o)}^c$.
\begin{theorem}\label{t9}
Let $K\in \Ksc$, $u\in C_c^\infty(\RN\setminus\{o\})$ and $a,b\in\R$.
Then the following max-gradient anisotropic inequalities hold.
\begin{align}\label{f547-max}
\int_{\RN}
\frac{\max\{\|\nabla u\|_{K^*}^2,\|-\nabla u\|_{K^*}^2\}}
{\hK{x}^{2b}}dx
+
\int_{\RN}\frac{u^2}{\hK{x}^{2a}}dx
\ge
C_{ckn1}(N,a,b)
\int_{\RN}\frac{u^2}{\hK{x}^{a+b+1}}dx,
\end{align}
where
$$
C_{ckn1}(N,a,b)=
\begin{cases}
|N-a-b-1|, & (a,b)\in \mathcal A,\\
|N-3b+a-3|, & (a,b)\in \mathcal B.
\end{cases}
$$
When $a=b+1$, the corresponding constant is
$$
C_{ckn1}(N,b+1,b)=1+\frac{(N-2b-2)^2}{4}.
$$
And
\begin{align}\label{f546-new}
\!\!\!\!\!\!\!\bigg(\!
\int_{\RN}\!\frac{\max\{\|\nabla u\|_{K^*}^2,\|-\nabla u\|_{K^*}^2\}}{\hK{x}^{2b}}dx
\bigg)^{\frac12}
\bigg(\!
\int_{\RN}\!\frac{u^2}{\hK{x}^{2a}}dx
\bigg)^{\frac12}
\!\ge\!
C_{ckn2}(N,a,b)
\int_{\RN}\!\frac{u^2}{\hK{x}^{a+b+1}}dx,
\end{align}
where
$$
C_{ckn2}(N,a,b)=
\begin{cases}
\big|\frac{N-a-b-1}{2}\big|, & (a,b)\in \mathcal A,\\[1mm]
\big|\frac{N-3b+a-3}{2}\big|, & (a,b)\in \mathcal B.
\end{cases}
$$
When $a=b+1$, the corresponding constant is
$$
C_{ckn2}(N,b+1,b)=\frac{|N-2(b+1)|}{2}.
$$
\end{theorem}
\begin{proof}
By the anisotropic Cauchy-Schwarz inequality \eqref{f22-0}, for  $x\in\RN\setminus\{o\}$,
$$
-\hK{x}\|-\nabla u(x)\|_{K^*}
\le
x\cdot\nabla u(x)
\le
\hK{x}\|\nabla u(x)\|_{K^*}.
$$
Therefore,
$
|\mathcal R_K(u)(x)|^2
\le
\max\{\|\nabla u(x)\|_{K^*}^2,\|-\nabla u(x)\|_{K^*}^2\}$
for a.e. $x\in\RN.
$
It follows that
\begin{align}\label{f80}
\int_{\RN}\frac{|\mathcal R_K(u)|^2}{\hK{x}^{2b}}dx
\le
\int_{\RN}
\frac{\max\{\|\nabla u\|_{K^*}^2,\|-\nabla u\|_{K^*}^2\}}
{\hK{x}^{2b}}dx.
\end{align}
Combining \eqref{f80} with the radial anisotropic Hardy-type identities
\eqref{f48}, \eqref{f76}, \eqref{f49} and \eqref{f50} gives
\eqref{f547-max} when $a\ne b+1$.
When $a=b+1$, combining  the anisotropic Hardy inequality \eqref{f69} with \eqref{f80}, one obtains
\begin{align*}
&\int_{\RN}
\frac{
\max\{\|\nabla u\|_{K^*}^2,\|-\nabla u\|_{K^*}^2\}
}{\hK{x}^{2b}}dx
+
\int_{\RN}
\frac{u^2}{\hK{x}^{2b+2}}dx\notag\\
&\ge
\bigg(
1+\frac{(N-2b-2)^2}{4}
\bigg)
\int_{\RN}
\frac{u^2}{\hK{x}^{2b+2}}dx.
\end{align*}

Similarly, combining \eqref{f80} with the radial anisotropic CKN
inequality \eqref{t-1} and Theorem \ref{t3} gives
\eqref{f546-new}. In particular, when $a=b+1$, Theorem
\ref{t3} ($iii$) gives the constant
$
\frac{|N-2b-2|}{2}.
$
\end{proof}

When $K$ is origin-symmetric, the max-gradient formulation in Theorem \ref{t9} gives the following norm-based anisotropic
gradient inequalities as a direct consequence.
\begin{corollary}
Let $K\in\Ksc\cap \mathcal{K}_{(o)}^c$, $u\in C_c^\infty(\RN\setminus\{o\})$ and $a,b\in\R$.
Then the following  anisotropic gradient  inequalities hold.
\begin{align}\label{f70}
\int_{\RN}
\frac{\|\nabla u\|_{K^*}^2}
{\hK{x}^{2b}}dx
+
\int_{\RN}\frac{u^2}{\hK{x}^{2a}}dx
\ge
C_{ckn1}(N,a,b)
\int_{\RN}\frac{u^2}{\hK{x}^{a+b+1}}dx,
\end{align}
where 
$$
C_{ckn1}(N,a,b)=
\begin{cases}
|N-a-b-1|, & (a,b)\in \mathcal A,\\
|N-3b+a-3|, & (a,b)\in \mathcal B.
\end{cases}
$$
And
\begin{align}\label{f72}
\!\!\!\!\!\!\!\bigg(\!
\int_{\RN}\!\frac{\|\nabla u\|_{K^*}^2}{\hK{x}^{2b}}dx
\bigg)^{\frac12}
\bigg(\!
\int_{\RN}\!\frac{u^2}{\hK{x}^{2a}}dx
\bigg)^{\frac12}
\!\ge\!
C_{ckn2}(N,a,b)
\int_{\RN}\!\frac{u^2}{\hK{x}^{a+b+1}}dx,
\end{align}
where
$$
C_{ckn2}(N,a,b)=
\begin{cases}
\big|\frac{N-a-b-1}{2}\big|, & (a,b)\in \mathcal A,\\[1mm]
\big|\frac{N-3b+a-3}{2}\big|, & (a,b)\in \mathcal B.
\end{cases}
$$
When $a=b+1$, the corresponding constants are
$$
C_{ckn1}(N,b+1,b)
=
1+\frac{(N-2b-2)^2}{4}
$$
and
$$
C_{ckn2}(N,b+1,b)
=
\frac{|N-2b-2|}{2}.
$$
Moreover, all the above constants are sharp.
\end{corollary}

\begin{proof}
Since $K\in\mathcal K_{(o)}^c$, the polar body $K^*$ is also
origin-symmetric. Hence, for every $\xi\in\RN$,
$
\|\xi\|_{K^*}=\|-\xi\|_{K^*}.
$
Therefore,
$$
\max\{\|\nabla u\|_{K^*}^2,\|-\nabla u\|_{K^*}^2\}
=
\|\nabla u\|_{K^*}^2
\quad
\text{a.e. in }\RN.
$$
Applying Theorem \ref{t9} gives \eqref{f70} and \eqref{f72}.

It remains to prove the sharpness of the constants. We first consider the
case $a\ne b+1$. Choosing $\Phi$ in the extremal functions from
Theorem \ref{t3} to be a nonzero constant and taking
$
t=-1$, $t=1$, $t=1$ and $t=-1
$
in $A_1$, $A_2$, $B_1$ and $B_2$, respectively, then the
extremal functions are radial with respect to $K$. By Theorem \ref{t3},
they attain equality in the corresponding radial multiplicative
inequalities. Moreover, direct substitution into the additive identities
\eqref{f48}, \eqref{f76}, \eqref{f49} and \eqref{f50} shows that the
corresponding square remainders vanish. Hence, they also attain equality
in the corresponding radial additive inequalities.

Indeed, let $r=\hK{x}$ and let $v(x)=v(r)$ be radial. Then
$
\nabla v(x)=v'(r)\nabla\hK{x}.
$
Since
$
x\cdot\nabla\hK{x}=\hK{x},
$
one has
$$
\mathcal R_K(v)(x)
=
\frac{x\cdot\nabla v(x)}{\hK{x}}
=
v'(r).
$$
Since
$
\|\nabla\hK{x}\|_{K^*}=1
$
and $\|\cdot\|_{K^*}$ is even, one obtains
$$
\|\nabla v(x)\|_{K^*}
=
|v'(r)|
=
|\mathcal R_K(v)(x)|
$$
for a.e. $x\in\RN$.

By the same radial cut-off and density argument used in the proof of
Theorem \ref{t3}, these radial extremal functions can be approximated by
functions in $C_c^\infty(\RN\setminus\{o\})$ in the corresponding
weighted full-gradient norm. Indeed, for the radial cut-offs, the
weighted full-gradient term coincides with the weighted
radial-derivative term by the identity above. For each fixed cut-off, the
final smooth approximation follows from the $W^{1,2}$-density theorem on
an annulus, where the weights are bounded and $\|\cdot\|_{K^*}$ is
equivalent to the Euclidean norm. Hence, the constants in the gradient
inequalities cannot be improved. This proves the sharpness when
$a\ne b+1$.

We next consider the case $a=b+1$.
Let $\{u_\varepsilon\}$ be the radial extremizing sequence constructed in
the proof of Theorem \ref{t3}, and let
$$
A_\varepsilon
=
\int_{\RN}
\frac{\|\nabla u_\varepsilon\|_{K^*}^2}
{\hK{x}^{2b}}dx\ \
\mathrm{and}\ \
B_\varepsilon
=
\int_{\RN}
\frac{u_\varepsilon^2}
{\hK{x}^{2b+2}}dx.
$$
Since $u_\varepsilon$ is radial, one has
$$
A_\varepsilon
=
\int_{\RN}
\frac{|\mathcal R_K(u_\varepsilon)|^2}
{\hK{x}^{2b}}dx.
$$
The construction in the proof of Theorem \ref{t3} gives
$$
\frac{A_\varepsilon^{\frac12}B_\varepsilon^{\frac12}}
{B_\varepsilon}
\rightarrow \frac{|N-2b-2|}{2}
\quad
\text{as }\varepsilon\to0.
$$
Consequently,
$$
\frac{A_\varepsilon}{B_\varepsilon}
\rightarrow \frac{(N-2b-2)^2}{4}
\quad
\text{as }\varepsilon\to0.
$$
Therefore,
$$
\frac{A_\varepsilon+B_\varepsilon}{B_\varepsilon}
\rightarrow
1+\frac{(N-2b-2)^2}{4}
\ \
\mathrm{and}\ \
\frac{A_\varepsilon^{\frac12}B_\varepsilon^{\frac12}}
{B_\varepsilon}
\rightarrow
\frac{|N-2b-2|}{2}
\quad
\text{as }\varepsilon\to0.
$$
Choosing the smooth approximation sufficiently close for each
$\varepsilon$ and then taking a diagonal sequence, we obtain extremizing
sequences in $C_c^\infty(\RN\setminus\{o\})$. Hence, the constants are
also sharp when $a=b+1$.
\end{proof}

\bigskip
\noindent \textbf{Acknowledgments.}
The author is grateful to Professors Baocheng Zhu and Denghui Wu  for their valuable comments and helpful suggestions.

\section*{Statements and Declarations}

The author declares no competing interests.


\begin{thebibliography}{99}

\bibitem{ACP}
B. Abdellaoui, E. Colorado and I. Peral, \textit{Some improved Caffarelli-Kohn-Nirenberg inequalities}, Calc. Var. Partial Differential Equations, \textbf{23} (2005), 327-345.

\bibitem{AFTL97}
A. Alvino, V. Ferone, G. Trombetti and P. L. Lions, \textit{Convex symmetrization and applications}, Ann. Inst. H. Poincar\'e{} C Anal. Non Lin\'eaire,
\textbf{14} (1997), 275-293.

\bibitem{BC25}
J. Bao and X. Chen,
\textit{On the anisotropic Caffarelli-Kohn-Nirenberg type inequalities:
Existence, symmetry breaking region and symmetry of extremal functions},
Commun. Contemp. Math., \textbf{27} (2025), Paper No. 2550016, 25 pp.

\bibitem{Bellettini}
G. Bellettini and M. Paolini, \textit{Anisotropic motion by mean curvature in the context of Finsler geometry}, Hokkaido Math. J., \textbf{25} (1996), 537-566.

\bibitem{BianchiCianchiGronchi}
G. Bianchi, A. Cianchi and P. Gronchi,
\textit{Anisotropic symmetrization, convex bodies, and isoperimetric inequalities},
Adv. Math., \textbf{462} (2025), Paper No. 110085, 36.

\bibitem{BE91}
G. Bianchi and H. Egnell,
\textit{A note on the Sobolev inequality},
J. Funct. Anal., \textbf{100} (1991), no. 1, 18-24.

\bibitem{BL85}
H. Brezis and E. H. Lieb,
\textit{Sobolev inequalities with remainder terms},
J. Funct. Anal., \textbf{62} (1985), no. 1, 73-86.

\bibitem{CKN}
L. Caffarelli, R. Kohn and L. Nirenberg, \textit{First order interpolation inequalities with
weights}, Compositio Math., \textbf{53} (1984), 259-275.

\bibitem{CM13}
P. Caldiroli and R. Musina, \textit{Symmetry breaking of extremals for the Caffarelli-Kohn-Nirenberg inequalities in a non-Hilbertian setting},
Milan J. Math., \textbf{81} (2013), 421-430.

\bibitem{CC09}
F. Catrina and D. G. Costa, \textit{Sharp weighted-norm inequalities for functions with compact support in $\RN\setminus \{0\}$}, J. Differential Equations, \textbf{246} (2009), 164-182.

\bibitem{CW01}
F. Catrina and Z. Wang, \textit{On the Caffarelli-Kohn-Nirenberg inequalities: sharp constants, existence (and nonexistence), and symmetry of extremal functions}, Comm. Pure Appl. Math., \textbf{54} (2001), 229-258.

\bibitem{CFL21}
C. Cazacu, J. Flynn and N. Lam, \textit{Short proofs of refined sharp Caffarelli-Kohn-Nirenberg inequalities}, J. Differential Equations, \textbf{302} (2021), 533-549.

\bibitem{Cazacu}
C. Cazacu, J. Flynn, N. Lam and G. Lu, \textit{Caffarelli-Kohn-Nirenberg identities, inequalities and their stabilities}, J. Math. Pures Appl. (9), \textbf{182} (2024), 253-284.

\bibitem{CLT24}
L. Chen, G. Lu and H. Tang,
\textit{Stability of Hardy-Littlewood-Sobolev inequalities with explicit lower bounds},
Adv. Math., \textbf{450} (2024), Paper No. 109778.

\bibitem{CLT242}
L. Chen, G. Lu and H. Tang,
\textit{Optimal asymptotic lower bound for stability of fractional Sobolev inequality
and the global stability of log-Sobolev inequality on the sphere},
Adv. Math., \textbf{479} (2025), Part B, Paper No. 110438.

\bibitem{CLT243}
L. Chen, G. Lu and H. Tang,
\textit{Optimal stability of Hardy-Littlewood-Sobolev and Sobolev inequalities
of arbitrary orders with dimension-dependent constants},
Math. Ann., \textbf{394} (2026), 77.

\bibitem{CLTW}
L. Chen, G. Lu, H. Tang and B. Wang,
\textit{Asymptotically sharp stability of Sobolev inequalities on the Heisenberg group
with dimension-dependent constants},
J. Math. Pures Appl., \textbf{206} (2026), Paper No. 103832.

\bibitem{CC22}
G. Ciraolo and R. Corso, \textit{Symmetry for positive critical points of Caffarelli-Kohn-Nirenberg inequalities},
Nonlinear Anal., \textbf{216} (2022), Paper No. 112683, 23 pp.

\bibitem{CNV}
D. Cordero-Erausquin, B. Nazaret and C. Villani, \textit{A mass-transportation approach to sharp Sobolev and Gagliardo-Nirenberg inequalities}, Adv. Math., \textbf{182} (2004), 307-332.

\bibitem{Cos08}
D. G. Costa, \textit{Some new and short proofs for a class of Caffarelli-Kohn-Nirenberg type inequalities}, J. Math. Anal. Appl., \textbf{337} (2008), 311-317.

\bibitem{DM23}
P. Dang and W. Mai,
\textit{Improved Caffarelli-Kohn-Nirenberg inequalities and uncertainty principle},
J. Geom. Anal., \textbf{34} (2024), Paper No. 70, 26.

\bibitem{DBG}
F. Della Pietra, G. di Blasio and N. Gavitone, \textit{Anisotropic Hardy inequalities}, Proc. Roy. Soc. Edinburgh Sect. A, \textbf{148} (2018), 483-498.

\bibitem{DD}
M. Del Pino and J. Dolbeault, \textit{Best constants for Gagliardo-Nirenberg inequalities and applications to nonlinear diffusions}, J. Math. Pures Appl.(9), \textbf{81} (2002), 847-875.

\bibitem{DT24}
S. Deng and X. Tian,
\textit{Gradient stability of Caffarelli-Kohn-Nirenberg inequality involving weighted $p$-Laplacian}, arXiv:2401.04129, 2024.

\bibitem{Do}
A. X. Do, J. Flynn, N. Lam and G. Lu, \textit{$L^p$-Caffarelli-Kohn-Nirenberg inequalities and their stabilities}, arXiv:2310.07083, 2023.

\bibitem{DoLam}
A. X. Do, N. Lam, G. Lu and V. H. Nguyen,
\textit{Caffarelli-Kohn-Nirenberg and weighted Gaussian Poincar\'e inequalities:
a complete characterization of sharp $L^2$ stability and $L^p$ extensions},
arXiv:2606.08939, 2026.

\bibitem{DEFFL}
J. Dolbeault, M. J. Esteban, A. Figalli, R. L. Frank and M. Loss,
\textit{Sharp stability for Sobolev and log-Sobolev inequalities, with optimal dimensional dependence},
Cambridge J. Math., \textbf{13} (2025), no. 2, 359-430.

\bibitem{DE12}
J. Dolbeault, M. J. Esteban and M. Loss, \textit{Symmetry of extremals of functional inequalities via spectral estimates for linear operators},
J. Math. Phys., \textbf{53} (2012), 095204, 18.

\bibitem{DET09}
J. Dolbeault, M. J. Esteban, M. Loss and G. Tarantello, \textit{On the symmetry of extremals for the Caffarelli-Kohn-Nirenberg inequalities},
Adv. Nonlinear Stud., \textbf{9} (2009), 713-726.

\bibitem{Dong18} M. Dong, \textit{Existence of extremal functions for higher-order Caffarelli-Kohn-Nirenberg inequalities}, Adv. Nonlinear Stud., \textbf{18} (2018), no. 3, 543-553.

\bibitem{DN25a}
A. T. Duong and V. H. Nguyen,
\textit{On the sharp second order Caffarelli-Kohn-Nirenberg inequality},
Ann. Fenn. Math., \textbf{50} (2025), no. 1, 275-286.

\bibitem{DN25b}
A. T. Duong and V. H. Nguyen,
\textit{On the stability estimate for the sharp second order uncertainty principle},
Calc. Var. Partial Differential Equations, \textbf{64} (2025), Paper No. 129.

\bibitem{DLL}
N. T. Duy, N. Lam and G. Lu,
\textit{$p$-Bessel pairs, Hardy’s identities and inequalities and Hardy-Sobolev inequalities with monomial weights},
J. Geom. Anal., \textbf{32} (2022), Paper No. 109, 36 pp.

\bibitem{DPH25}N. T. Duy, N. V. Phong and P. T. T. Hien,
\textit{Hardy inequalities with Bessel pair for Dunkl operator},
Adv. Nonlinear Stud., \textbf{25} (2025), no. 4, 1127-1141.

\bibitem{FS}
V. Felli and M. Schneider,
\textit{Perturbation results of critical elliptic equations of Caffarelli-Kohn-Nirenberg type},
J. Differential Equations, \textbf{191} (2003), 121-142.

\bibitem{Ferone-Volpicelli}
A. Ferone and R. Volpicelli, \textit{Convex rearrangement: equality cases in the P\'olya-Szeg\"o inequality}, Calc. Var. Partial Differential Equations, \textbf{21} (2004), 259-272.

\bibitem{FigalliMaggiPratelli}
A. Figalli, F. Maggi and A. Pratelli,
\textit{A mass transportation approach to quantitative isoperimetric inequalities},
Invent. Math., \textbf{182} (2010), no. 1, 167-211.

\bibitem{FMP}
A. Figalli, F. Maggi and A. Pratelli,
\textit{Sharp stability theorems for the anisotropic Sobolev and log-Sobolev inequalities on functions of bounded variation},
Adv. Math., \textbf{242} (2013), 80-101.

\bibitem{Flynn20} J. Flynn, \textit{Sharp Caffarelli-Kohn-Nirenberg-type inequalities on Carnot groups}, Adv. Nonlinear Stud., \textbf{20} (2020), no. 1, 95-111.

\bibitem{FonsecaMuller}
I. Fonseca and S. Müller,
\textit{A uniqueness proof for the Wulff theorem},
Proc. Roy. Soc. Edinburgh Sect. A, \textbf{119} (1991), no. 1-2, 125-136.

\bibitem{GM11}
N. Ghoussoub and A. Moradifam,
\textit{Bessel pairs and optimal Hardy and Hardy-Rellich inequalities},
Math. Ann., \textbf{349} (2011), 1-57.

\bibitem{GM13}
N. Ghoussoub and A. Moradifam,
\textit{Functional inequalities: new perspectives and new applications},
Math. Surveys Monogr., \textbf{187} (2013), xxiv+299.

\bibitem{Lam19}
N. Lam, \textit{General sharp weighted Caffarelli-Kohn-Nirenberg inequalities}, Proc. Roy. Soc. Edinburgh Sect. A, \textbf{149} (2019), 691-718.

\bibitem{Lam21}
N. Lam, \textit{Sharp weighted isoperimetric and Caffarelli-Kohn-Nirenberg inequalities}, Adv. Calc. Var., \textbf{14} (2021), 153-169.

\bibitem{LamLu17}
N. Lam and G. Lu, \textit{Sharp constants and optimizers for a class of Caffarelli-Kohn-Nirenberg inequalities},
Adv. Nonlinear Stud., \textbf{17} (2017), 457-480.

\bibitem{LMP2019}
N. Lam, A. Maalaoui and A. Pinamonti,
\textit{Characterizations of anisotropic high order Sobolev spaces},
Asymptot. Anal., \textbf{113} (2019), no. 4, 239-260.

\bibitem{LiYan23}
Y. Li and X. Yan,
\textit{Anisotropic Caffarelli-Kohn-Nirenberg type inequalities},
Adv. Math., \textbf{419} (2023), Paper No. 108958, 44.

\bibitem{Lieb}
E. H. Lieb, \textit{Sharp constants in the Hardy-Littlewood-Sobolev and related inequalities}, Ann. of Math. (2), \textbf{118} (1983), 349-374.

\bibitem{LXZZZ26}
Q. Liu, J. Xiao, N. Zhang, R. Zhang and B. Zhu,
\textit{A Minkowski theory for the exterior capacitary volumes and a resolution
of the P\'olya-Szeg\"{o} conjecture},
arXiv:2607.02273, 2026.

\bibitem{Lu}
G. Lu, Y. Shen, J. Xue and M. Zhu, \textit{Weighted anisotropic isoperimetric inequalities and existence of extremals for singular anisotropic Trudinger-Moser inequalities}, Adv. Math., \textbf{458} (2024), Paper No. 109949, 51 pp.

\bibitem{Mercaldo}
A. Mercaldo, M. Sano and F. Takahashi, \textit{Finsler Hardy inequalities}, Math. Nachr., \textbf{293} (2020), 2370-2398.

\bibitem{Nguyen}
V. H. Nguyen, \textit{New approach to the affine P\'olya-Szeg\"o principle and the stability version of the affine Sobolev inequality}, Adv. Math., \textbf{302} (2016), 1080-1110.

\bibitem{N}
V. H. Nguyen, \textit{Sharp weighted Sobolev and Gagliardo-Nirenberg inequalities}, Proc. Lond. Math. Soc., \textbf{111} (2015), 1276-1321.

\bibitem{Shen25}
Y. Shen, \textit{On the anisotropic Caffarelli-Kohn-Nirenberg and weighted Hardy-Sobolev type inequalities: Sharp constants, existence and explicit form of extremal functions},
Discrete Contin. Dyn. Syst., to appear, 2025.

\bibitem{Talenti}
G. Talenti, \textit{Best constant in Sobolev inequality}, Ann. Mat. Pura Appl.(4), \textbf{110} (1976), 353-372.

\bibitem{Taylor}
J. E. Taylor,
\textit{Crystalline variational problems},
Bull. Amer. Math. Soc., \textbf{84} (1978), no. 4, 568-588.

\bibitem{VanSchaftingen}
J. Van Schaftingen, \textit{Anisotropic symmetrization}, Ann. Inst. H. Poincar\'e Anal. Non Lin\'eaire, \textbf{23} (2006), no. 4, 539-565.

\bibitem{WW03}
Z.-Q. Wang and M. Willem, \textit{Caffarelli-Kohn-Nirenberg inequalities with remainder terms},
J. Funct. Anal., \textbf{203} (2003), 550-568.

\end{thebibliography}
\end{document}